\documentclass[11pt]{amsart}
\usepackage{amscd,amssymb}
\usepackage{amsthm,amsmath,amssymb}
\usepackage[matrix,arrow]{xy}

\newtheorem{theorem}[equation]{Theorem}

\newtheorem{proposition}[equation]{Proposition}
\newtheorem{lemma}[equation]{Lemma}
\newtheorem{corollary}[equation]{Corollary}
\newtheorem{conjecture}[equation]{Conjecture}

\theoremstyle{definition}
\newtheorem{example}[equation]{Example}
\newtheorem{definition}[equation]{Definition}

\theoremstyle{remark}
\newtheorem{remark}[equation]{Remark}

\makeatletter\@addtoreset{equation}{section} \makeatother

\renewcommand{\thefootnote}{}

\begin{document}
%\udc{512.761.5}

\title{On Fano threefolds with canonical Gorenstein singularities}\footnote{The work was partially supported by
RFFI grant No. 08-01-00395-a and grant N.Sh.-1987.2008.1.}

\author{Ilya Karzhemanov}

\begin{abstract}
We classify three-dimensional Fano varieties with canonical Gorenstein singularities of degree bigger than $64$.
\end{abstract}

\sloppy

\maketitle

\section{Introduction}
\label{section:introduction}

\renewcommand{\thefootnote}{1}
In the present article we study three-dimensional Fano varieties\footnote{We assume that all varieties
are projective, normal and  defined over $\mathbb{C}$.} with canonical Gorenstein singularities. We consider
such varieties from the view point of the value of their anti-canonical degree.

It follows from the main result in \cite{Borisov-boundedness} that the anti-canonical degree of Fano threefolds with
canonical Gorenstein singularities is bounded. Under the stronger restrictions on singularities there is a precise estimate:

\begin{theorem}[\cite{VA-1}, \cite{VA-2}, \cite{Mori-Mukai}, \cite{Lvovsky}]
\label{theorem:smooth-case}
Let $X$ be a non-singular Fano threefold. Then $-K_{X}^{3} \leqslant 64$ with equality only for $\mathbb{P}^{3}$.
\end{theorem}

For Fano threefolds with terminal Gorenstein singularities we have the following result:

\begin{theorem}[\cite{Namikawa-smoothing}]
\label{theorem:Namikawa-smoothing}
Let $X$ be a Fano threefold with terminal Gorenstein singularities. Then $X$ can be deformed to a non-singular Fano threefold.
In particular, $-K_{X}^{3} \leqslant 64$.
\end{theorem}

Much more general result in the direction of getting a sharp bound for the degree of Fano threefolds with canonical
Gorenstein singularities was obtained in \cite{Prokhorov-degree}:

\begin{theorem}
\label{theorem:Prokhorov-degree}
Let $X$ be a Fano threefold with canonical Gorenstein singularities. Then  $-K_{X}^{3} \leqslant 72$ with equality only for
$\mathbb{P}(3,1,1,1)$ and $\mathbb{P}(6,4,1,1)$.
\end{theorem}

It was conjectured in the paper \cite{Prokhorov-degree} that for the Fano threefold $X$ with only
$\mathrm{cDV}$ singularities the estimate $-K_{X}^{3} \leqslant 64$ holds. In the present article we will prove the
following

\begin{theorem}
\label{theorem:main}
Let $X$ be a Fano threefold with canonical Gorenstein singularities such that $64 < -K_{X}^{3} \leqslant 70$. Then $X$ is one of the following:
\begin{itemize}
\item The image under birational projection of the anti-canonically embedded Fano threefold $\mathbb{P}(6,4,1,1) \subset \mathbb{P}^{38}$
from a singular point on $\mathbb{P}(6,4,1,1)$ of type $\mathrm{cDV}$. In this case $-K_{X}^{3} = 70$ and $X$ has non-$\mathrm{cDV}$ singularities;

\item The anti-canonical image of the rational scroll
$\mathrm{Proj}\left(\mathcal{O}_{\mathbb{P}^{1}}(5)\oplus\mathcal{O}_{\mathbb{P}^{1}}(2)\oplus\mathcal{O}_{\mathbb{P}^{1}}\right)$. In this case
$-K_{X}^{3} = 66$ and $X$ has non-$\mathrm{cDV}$ singularities.
\end{itemize}
\end{theorem}

\begin{remark}
Threefold $X$ with $-K_{X}^{3} = 70$ was found by I. Cheltsov.
\end{remark}

It follows easily from the Riemann--Roch formula on the Fano threefold $X$ that $-K_{X}^{3}$ is an even number. Thus,
Theorems~\ref{theorem:Prokhorov-degree}
and \ref{theorem:main} exaust all Fano threefolds with canonical Gorenstein singularities of degree bigger than $64$.
Moreover, since threefolds $\mathbb{P}(3,1,1,1)$ and $\mathbb{P}(6,4,1,1)$ have singularities worse than $\mathrm{cDV}$,
from Theorems~\ref{theorem:Prokhorov-degree} and \ref{theorem:main} we deduce

\begin{corollary}
\label{theorem:estimate-for-cDV-Fano}
Let $X$ be a Fano threefold with only $\mathrm{cDV}$ singularities. Then $-K_{X}^{3} \leqslant 64$.
\end{corollary}

Corollary~\ref{theorem:estimate-for-cDV-Fano} appears to be a justification for the following conjecture in
\cite{Prokhorov-degree}:

\begin{conjecture}
\label{theorem:conjecture-about-smoothing}
Let $X$ be a Fano threefold with only $\mathrm{cDV}$ singularities. Then $X$ has a smoothing by a flat deformation.
\end{conjecture}

Let us give the following

\begin{definition}
Threefold $U$ with canonical singularities is called a \emph{Fano--Enriques threefold} if $-K_{U}$ is not Cartier but
$-K_{U} \equiv H$ for some ample Cartier divisor $H$.
\end{definition}

\begin{remark}
\label{remark:Fano-Enriques} The study of threefolds possessing an
ample Cartier divisor which is a non-singular Enriques surface was
initiated by G. Fano in \cite{Fano} (see also \cite{Conte} and
\cite{Conte-Murre}). It is easy to see (see for example
\cite{Cheltsov-Enriques}) that if such a threefold has canonical
singularities, then it is a Fano--Enriques threefold. On the other
hand, according to \cite[Proposition 3.1]{Prokhorov-Enriques-1}
and \cite[Proposition 3.3]{Prokhorov-Enriques-1} general element
$S \in |H|$ on Fano--Enriques threefold $U$ has only Du Val
singularities and the minimal resolution of $S$ is an Enriques
surface. Moreover, if $U$ has isolated singularities and $H^{3}
\ne 2$, then general element in $|H|$ is a non-singular Enriques
surface (see \cite[Corollary 3.7]{Prokhorov-Enriques-1}).

According to \cite{Cheltsov-Enriques} the equivalence $2(K_{U}+S)
\sim 0$ holds. Thus, on $U$ we have $-K_{U} \sim_{\mathbb{Q}} H$.
Moreover, by \cite[Proposition 14]{Cheltsov-Enriques-1} if $-K_{U}
\sim_{\mathbb{Q}} nH'$ for some integer $n > 0$ and ample Cartier
divisor $H'$, then $n = 1$. In particular, $\mathbb{Q}$-Fano
threefolds of Fano index $1$ form a subclass of Fano--Enriques
threefolds. $\mathbb{Q}$-Fano threefolds of Fano index $1$ with
cyclic quotient singularities were classified in \cite{Bayle} and
\cite{Sano}.
\end{remark}

Let $U$ be a Fano--Enriques threefold such that $-K_{U} \equiv H$
for some Cartier divisor $H$. By \cite[Lemma
2.2]{Prokhorov-Enriques-1} the value $g: = \frac{1}{2}H^{3} + 1$
is an integer number. Let us call it the \emph{genus} of $U$. The
problem of genus bound for Fano--Enriques threefolds was
investigated in \cite{Prokhorov-Enriques} and \cite{KLM}. In
\cite{Prokhorov-Enriques} the sharp bound $g \leqslant 17$ was
proved under the assumption that $U$ has only isolated
singularities.

In our setting we have

\begin{corollary}
\label{theorem:genus-of-FE-threefold}
Let $U$ be  a Fano--Enriques threefold of genus $g$. Then $g \leqslant 17$.
\end{corollary}

\begin{proof}
Since $2(K_{U} + S) \sim 0$ for $S \in |H|$ (see
Remark~\ref{remark:Fano-Enriques}), we can take a global log
canonical cover $\pi : X \longrightarrow U$ with respect to $K_{U}
+ S$ (see \cite{Kawamata}). Morphism $\pi$ has degree $2$ and the
equivalence $\pi^{*}(K_{U}+S) \sim 0$ holds. In particular, $X$ is
a Fano threefold with canonical Gorenstein singularities and
degree $-K_{X}^{3} = 4g - 4$. Moreover, $\pi$ is ramified
precisely in that points on $U$ where $K_{U}$ is not Cartier.
Since $-K_{U} \sim_{\mathbb{Q}} S$ and $S$ is normal, the set of
such points is finite. Hence the Galous involution of $\pi$
induces on $X$ an automorphism $\tau$ of order $2$ such that
$\tau$ acts freely in codimension $2$.

Suppose that $g > 17$. Then we get $-K_{X}^{3} > 64$. Since
$-K_{X}^{3}$ is divisible by $4$, from
Theorems~\ref{theorem:Prokhorov-degree} and \ref{theorem:main} we
deduce that $X = \mathbb{P}(3,1,1,1)$ or $\mathbb{P}(6,4,1,1)$.

If $X = \mathbb{P}(3,1,1,1)$, then the linear system
$|-\frac{1}{2}K_{X}|$ gives an embedding of $X$ in
$\mathbb{P}^{10}$ such that the image of $X$ is the cone over del
Pezzo surface of degree $9$. Moreover, since $|-\frac{1}{2}K_{X}|$
is $\tau$-invariant, there is a $\tau$-invariant hyperplane
section $S$ of $X \subset \mathbb{P}^{10}$ not passing through the
vertex. Then $S \simeq \mathbb{P}^{2}$ and hence $\tau$ fixes
infinitely many points on $S$ which is impossible.

Thus, $X = \mathbb{P}(6,4,1,1)$. Take weighted projective
coordinates $x_{0}$, $x_{1}$, $x_{2}$, $x_{3}$ on $X$ such that
$\deg x_{0} = 6$, $\deg x_{1} = 4$, $\deg x_{2} = \deg x_{3}  = 1$
and $x_{i}$ is an eigen function with an eigen value $\pm 1$ with
respect to $\tau$, $1 \leqslant i \leqslant 4$. If both $x_{0}$,
$x_{1}$ ($x_{2}$, $x_{3}$, respectively) have the same eigen
value, then the curve given by equations $x_{2}=x_{3}=0$
($x_{0}=x_{1}=0$, respectively) consists of $\tau$-fixed points
which is impossible. Thus, we can assume that $\tau$ is given by
the following formula
$$
[x_{0}:x_{1}:x_{2}:x_{3}] \mapsto [x_{0}:-x_{1}:x_{2}:-x_{3}].
$$
But again the curve given by equations $x_{1} = x_{3} = 0$
consists of $\tau$-fixed points. The obtained contradiction
completes the proof of
Corollary~\ref{theorem:genus-of-FE-threefold}.
\end{proof}

We prove Theorem~\ref{theorem:main} following the ideas of \cite{Prokhorov-degree}. To be more precise, for the given Fano
threefold $X$ satisfying the conditions of Theorem~\ref{theorem:main}
we take a \emph{terminal $\mathbb{Q}$-factorial modification} $\varphi: Y \longrightarrow X$. Here morphism $\varphi$ is
crepant and threefold $Y$ has terminal factorial singularities. In particular, $-K_{Y}^{3} = -K_{X}^{3}$ and thus all the
considerations can be lifted to $Y$.

For $Y$ the complete description of $K$-negative extremal contractions can be obtained. With this description at hand in
Section~\ref{section:begin} we find $X$ with $-K_{X}^{3} = 70$ and non-$\mathrm{cDV}$ singularities
and reduce the proof of Theorem~\ref{theorem:main} to the
case when there are no $K$-negative extremal contractions of $Y$ to threefolds with terminal factorial singularities
having nef and big anti-canonical divisor. This and results in \cite{Prokhorov-degree} imply that the initial Fano threefold
$X$ is singular along a line or contains a plane.

In Section~\ref{section:reduction} following
\cite{Prokhorov-degree} we obtain a Mori fibre space $W$ with a
birational map $W \dashrightarrow X$ given by the linear subsystem
in $|-K_{W}|$. It is straightforward from \cite{Prokhorov-degree}
that $W$ is not a $\mathbb{Q}$-Fano threefold.

In Section~\ref{section:curve} for $W$ a del Pezzo fibration we find $X$ with $-K_{X}^{3} = 66$ and non-$\mathrm{cDV}$ singularities
and reduce the proof of
Theorem~\ref{theorem:main} to the case when $W$ is a conic bundle. Following \cite{Prokhorov-degree} we will exclude this
case in Section~\ref{section:surface}.
\\

I would like to thank I. Cheltsov, Yu. G. Prokhorov and C. Shramov for their numerous remarks and attention to the present paper which led to
the correction of the gap in the older version. I also would like to express my gratitude to V. A. Iskovskikh for stimulating discussions.

\section{Basic notions and statements}
\label{section:basic}

We use standard notions and facts from the Minimal Model Program (see for example \cite{Kollar-Mori}).
Recall that algebraic variety $X$ is called a \emph{Fano variety} if divisor $-K_{X}$ is an ample
$\mathbb{Q}$-Cartier divisor. If divisor $-K_{X}$ is only nef and big, then $X$ is called a \emph{weak Fano variety}.

In the present article we consider three-dimensional (weak) Fano varieties only with (no more than) canonical Gorenstein singularities.

Let $X$ be a Fano threefold.
From the Riemann--Roch Formula and Kawamata--Vieweg Vanishing Theorem it is easy to find the dimension of the anti-canonical
linear system on $X$ (see \cite{Reid-morphisms
-Kawamata}):
$$
\dim |-K_{X}| = -\frac{1}{2}K_{X}^{3} + 2
$$
The constant $-K_{X}^{3}$ is called the \emph{degree} of $X$ and $g:= -\frac{1}{2}K_{X}^{3} + 1$
is the \emph{genus} of $X$. Thus, $\dim |-K_{X}| = g + 1$ and $-K_{X}^{3} = 2g-2$. In particular, $-K_{X}^{3}$ is an even
number and for $X$ satisfying the conditions of Theorem~\ref{theorem:main} we have $-K_{X}^{3} \in \{66,68,70\}$.

Furthermore, we have

\begin{theorem}[\cite{Reid-morphisms-Kawamata}]
\label{theorem:elefant}
A general member $L \in |-K_{X}|$ has only Du Val singularities.
\end{theorem}

\begin{proposition}[\cite{Jahnke-Radloff}]
\label{theorem:non-free-antican-system}
In the above notation, if $\mathrm{Bs}|-K_{X}| \ne \emptyset$, then $-K_{X}^{3} \leqslant 22$.
\end{proposition}

For the movable anti-canonical linear system $|-K_{X}|$ let us denote by
$\varphi_{|-K_{X}|}: X \dashrightarrow \mathbb{P}^{g+1}$ the corresponding rational map.

\begin{proposition}[{\cite[Theorem 1.5]{CPS}}]
\label{theorem:free-antican-system-1}
In the above notation, if $\varphi_{|-K_{X}|}$ is a morphism which is not an embedding, then $-K_{X}^{3} \leqslant 40$.
\end{proposition}

\begin{proposition}[{\cite[Theorem 1.6]{CPS}}]
\label{theorem:free-antican-system-2}
In the above notation, if $\varphi_{|-K_{X}|}$ is an embedding and
$\varphi_{|-K_{X}|}(X)$ is not an intersection of quadrics, then $-K_{X}^{3} \leqslant 54$.
\end{proposition}

\begin{remark}
\label{remark:anti-canonical-embedding}
From Propositions~\ref{theorem:non-free-antican-system}--\ref{theorem:free-antican-system-2} we deduce that for the given Fano
threefold $X$ with $-K_{X}^{3} \in \{66,68,70\}$ the linear system $|-K_{X}|$ is free and determines an embedding
$\varphi_{|-K_{X}|}: X \longrightarrow \mathbb{P}^{g+1}$ with $g \in \{34,35,36\}$ such that
$X_{2g-2}:= \varphi_{|-K_{X}|}(X)$ is an intersection of quadrics.
\end{remark}

Now let us consider several birational properties of the given Fano threefold $X$.

\begin{theorem}[\cite{Kollar-Mori}]
\label{theorem:terminal-modification}
Let $V$ be a threefold with only canonical singularities. Then there exists a threefold $W$ with only terminal
$\mathbb{Q}$-factorial singularities and a birational contraction $\phi: W \longrightarrow V$ such that
$K_{W} = \phi^{*}K_{V}$. Such $W$ (or $\phi$) is called a \emph{terminal $\mathbb{Q}$-factorial modification of $V$}.
\end{theorem}

\begin{remark}
\label{remark:terminal-modifications-are-connected-by-flops}
In the notation of Theorem~\ref{theorem:terminal-modification}, let $W$ and $W'$ be two terminal $\mathbb{Q}$-factorial
modifications of $V$. Then $W$ and $W'$ are relative minimal models over $V$ (see \cite{Kollar-Mori}). Thus, the birational
map $W \dashrightarrow W'$ is either an isomorphism or a composition of flops over $V$ (see \cite[Theorem 4.3]{Kollar-flops}).
\end{remark}

Applying Theorem~\ref{theorem:terminal-modification} to $X$ we obtain a birational contraction $\varphi: Y \longrightarrow X$
such that $K_{Y} = \varphi^{*}K_{X}$ and $Y$ is a weak Fano threefold with terminal Gorenstein $\mathbb{Q}$-factorial
singularities and $-K_{Y}^{3} = -K_{X}^{3}$.

Furthermore, we have

\begin{lemma}[{\cite[Lemma 5.1]{Kawamata}}]
\label{theorem:factoriality}
In the above notation, threefold $Y$ has factorial singularities.
\end{lemma}

\begin{remark}
\label{remark:K-trivial-contraction-1}
Conversely, for any weak Fano threefold $Y$ with terminal factorial singularities its multiple anti-canonical image
$X = \varphi_{|-nK_{Y}|}(Y)$ for some $n \in \mathbb{N}$ is a Fano threefold such that $K_{Y} = \varphi^{*} K_{X}$.
In particular, $X$ has canonical Gorenstein singularities (see \cite{Kawamata}).
\end{remark}

\begin{proposition}[{\cite[Lemmas 4.2, 4.3]{Prokhorov-degree}}]
\label{theorem:extremal-rays}
In the above notation, the Mori cone $\overline{NE}(Y)$ is polyhedral and generated by contractible extremal rays (i.e., rays $R$ for which there
exists a contraction $\varphi_{R}: Y \longrightarrow Y'$ onto the normal variety $Y'$ such that the image of a curve $C$ is
a point if and only if $[C] \in R$ for the class of numerical equivalence of $C$).
\end{proposition}

\begin{remark}
\label{remark:K-trivial-contraction-2}
From Theorem~\ref{theorem:Namikawa-smoothing} for the given Fano threefold $X$ with $-K_{X}^{3} \in \{66,68,70\}$
we deduce that $X \ne Y$. In particular, we have $\rho(Y) > 1$ for the Picard number of $Y$. Hence according to
Proposition~\ref{theorem:extremal-rays}
the anti-canonical divisor $-K_{Y}$ defines a face of the Mori cone $\overline{NE}(Y)$ and $\varphi: Y \longrightarrow X$ is
the contraction of this face.  \end{remark}

\begin{example}
\label{example:examp-1} Let us consider the
$\mathbb{P}^{1}$-bundle $Y =
\mathrm{Proj}(\mathcal{O}_{\mathbb{P}^{2}}\oplus\mathcal{O}_{\mathbb{P}^{2}}(3))$
over $\mathbb{P}^{2}$. For the tautological line bundle
$\mathcal{O}_{Y}(1)$ the linear system $|\mathcal{O}_{Y}(1)|$
determines a birational contraction $\varphi: Y \longrightarrow X$
of the negative section on $Y$. Since $K_{Y} \sim
\mathcal{O}_{Y}(2)$, we have $K_{Y} = \varphi^{*}K_{X}$. This
implies that $X$ is a Fano threefold with canonical Gorenstein
singularities (see Remark~\ref{remark:K-trivial-contraction-1})
and $Y$ is a terminal $\mathbb{Q}$-factorial modification of $X$.
Moreover, by construction $X$ has index $2$ and degree $72$. It
then follows from Theorem~\ref{theorem:Prokhorov-degree}  that $X
\simeq \mathbb{P}(3,1,1,1)$.
\end{example}

\begin{remark}
\label{remark:unique-terminal-modification-1} In the notation of
Example~\ref{example:examp-1}, $\varphi$ is an extremal
contraction and $\varphi$-exceptional locus is isomorphic to
$\mathbb{P}^{2}$. This in particular implies that there are no
small $K$-trivial extremal contractions on $Y$. Then according to
Remark~\ref{remark:terminal-modifications-are-connected-by-flops}
all terminal $\mathbb{Q}$-factorial modifications of
$\mathbb{P}(3,1,1,1)$ are isomorphic to $Y$.
\end{remark}

\begin{example}
\label{example:examp-2} Let us consider the weighted projective
space $X = \mathbb{P}(6,4,1,1)$. The singular locus of $X$ is a
curve $L \simeq \mathbb{P}^{1}$ such that for two points $P$ and $Q$
on $L$ singularities $P \in X$, $Q \in X$ are of types
$\frac{1}{6}(4,1,1)$, $\frac{1}{4}(2,1,1)$, respectively, and for
every point $O \in L \setminus\{P,Q\}$ singularity $O \in X$ is
analytically isomorphic to $(0, o) \in \mathbb{C}^{1} \times U$,
where $o \in U$ is a singularity of type $\frac{1}{2}(1,1)$ (see
\cite[5.15]{Iano-Fletcher}). Since $-K_{X} \sim \mathcal{O}_{X}(12)$
is ample, this implies that $X$ is a Fano threefold with canonical
Gorenstein singularities.

Let $\sigma_{1}: Y_{1} \longrightarrow X$ be the blow up at the point $P$ with weights $\frac{1}{6}(4,1,1)$. Then the
singular locus of the threefold $Y_{1}$ is a curve $L_{1}$ such that
for two points $P_{1}$ and $Q_{1}$ on $L_{1}$ singularities $P_{1} \in Y_{1}$, $Q_{1} \in Y_{1}$ are of type $\frac{1}{4}(2,1,1)$
and for every point $O \in L_{1} \setminus\{P_{1},Q_{1}\}$ singularity $O \in Y_{1}$ is analytically isomorphic to
$(0, o) \in \mathbb{C}^{1} \times U$, where $o \in U$ is a singularity of type $\frac{1}{2}(1,1)$.

Let $\sigma_{2}: Y_{2} \longrightarrow Y_{1}$ be the blow up at the points $P_{1}$ and $Q_{1}$ with weights
$\frac{1}{4}(2,1,1)$. Then the singular locus of the threefold $Y_{2}$ is a curve $L_{2} \simeq \mathbb{P}^{1}$ such that
for every point $O$ on $L_{2}$ singularity $O \in Y_{2}$ is analytically isomorphic to
$(0, o) \in \mathbb{C}^{1} \times U$, where $o \in U$ is a singularity of type $\frac{1}{2}(1,1)$.

Finally, the blow up $\sigma_{3}: Y \longrightarrow Y_{2}$ of the ideal of the curve $L_{2}$ on $Y_{2}$ leads to the
non-singular threefold $Y$ and the birational contraction $\varphi: Y \longrightarrow X$. By construction morphisms
$\sigma_{i}$ are crepant, $1\leqslant i\leqslant 3$, and morphism $\varphi$ is a composition of $\sigma_{i}$.
This implies that $K_{Y} = \varphi^{*}K_{X}$ and $Y$ is a terminal $\mathbb{Q}$-factorial modification of $X$.
\end{example}

\begin{remark}
\label{remark:unique-terminal-modification-2} In the notation of
Example~\ref{example:examp-1}, $\varphi$ is a composition of
extremal contractions and $\varphi$-exceptional locus is of pure
codimension $1$. This in particular implies that there are no
small $K$-trivial extremal contractions on $Y$. Then according to
Remark~\ref{remark:terminal-modifications-are-connected-by-flops}
all terminal $\mathbb{Q}$-factorial modifications of
$\mathbb{P}(6,4,1,1)$ are isomorphic to $Y$.
\end{remark}

In the above notation, let us consider a $K$-negative extremal contraction $f: Y \longrightarrow Y'$. We have

\begin{proposition}[{\cite[Propositions 4.11, 5.2]{Prokhorov-degree}}]
\label{theorem:non-birational-contraction}
If $\dim Y' < \dim Y$, then $-K_{X}^{3} \notin \{66,68,70\}$.
\end{proposition}

Let us assume now that contraction $f$ is birational. Then by
\cite[(2.3.2)]{Mori-flip} morphism $f$ is divisorial. Let $E$ be
the $f$-exceptional divisor.

From the classification of extremal rays on terminal factorial
threefolds in \cite{Cutkosky} we obtain

\begin{lemma}
\label{theorem:0-contraction}
In the above notation, if $f(E)$ is a point, then one of the following holds:
\begin{itemize}
\item $Y'$ has only terminal factorial singularities and is a weak Fano threefold with $-K_{Y'}^{3} > -K_{Y}^{3}$;
\item $E \simeq \mathbb{P}^{2}$, $\mathcal{O}_{E}(E) \simeq \mathcal{O}_{\mathbb{P}^{2}}(2)$ and $f(E) \in Y'$ is a point of
type $\frac{1}{2}(1,1,1)$.
\end{itemize}
\end{lemma}

\begin{corollary}
\label{theorem:0-cor-contraction}
In the assumptions of Lemma~\ref{theorem:0-contraction},
if $f(E) \in Y'$ is a point of
type $\frac{1}{2}(1,1,1)$, then $\varphi(E)$ is a plane on $X$ (a surface $\Pi \simeq \mathbb{P}^{2}$ such that
$K_{X}^{2} \cdot \Pi = 1$).
\end{corollary}

\begin{proof}
Since $E \simeq \mathbb{P}^{2}$ and $\mathcal{O}_{E}(E) \simeq \mathcal{O}_{\mathbb{P}^{2}}(2)$, we have
$$
K_{Y}^{2} \cdot E = \left(f^{*}K_{Y'} + \frac{1}{2}E\right)^{2} \cdot E = \frac{1}{4}\left(\mathcal{O}_{\mathbb{P}^{2}}(2)\right)^{2} = 1.
$$
This implies that $\varphi(E)$ is a surface $\Pi \simeq \mathbb{P}^{2}$ on $X$ with $K_{X}^{2} \cdot \Pi = 1$.
\end{proof}

\begin{lemma}
\label{theorem:1-contraction}
In the above notation, if $f(E)$ is a curve, then $Y'$ has only terminal factorial singularities and
one of the following holds:
\begin{itemize}
\item $Y'$ is a weak Fano threefold with $-K_{Y'}^{3} \geqslant -K_{Y}^{3}$;
\item For $C:=f(E)$ we have $K_{Y'} \cdot C > 0$ and $C$ is the only curve having positive intersection with $K_{Y'}$. In
this case $C \simeq \mathbb{P}^{1}$ and either $E \simeq \mathbb{P}^{1} \times \mathbb{P}^{1}$ or $\mathbb{F}_{1}$.
\end{itemize}
\end{lemma}

\begin{proof}[Sketch of the proof]
We use the arguments from the proof of Proposition-definition 4.5
in \cite{Prokhorov-degree}. According to \cite{Cutkosky} threefold
$Y'$ is non-singular near $C$ and $f$ is the blow up of $C$. This
in particular implies that $Y'$ has only terminal factorial
singularities.

We have $K_{Y}=f^{*}K_{Y'}+E$ and
$$
K_{Y}^{3}=K_{Y'}^{3}+3f^{*}K_{Y'} \cdot E^{2} + E^{3}.
$$

If $Y'$ is a weak Fano threefold, then
$$
0 \leqslant \left(-K_{Y}\right)^{2} \cdot E = 2f^{*}K_{Y'} \cdot E^{2} + E^{3}, \qquad
0 \leqslant \left(-K_{Y}\right)\cdot\left(-f^{*}K_{Y'}\right)\cdot E = f^{*}K_{Y'} \cdot E^{2}.
$$
This gives the inequality $-K_{Y'}^{3} \geqslant -K_{Y}^{3}$.

Now, if $Y'$ is not a weak Fano threefold, then $K_{Y'} \cdot Z > 0$ for some irreducible curve $Z$. It is easy
to see that $Z = C$. Then from Lemmas 4.2 and 4.3 in \cite{Prokhorov-degree} and Corollary 1.3 in \cite{Mori-flip} we get
$C \simeq \mathbb{P}^{1}$.

Thus, $E \simeq \mathbb{F}_{n}$ for some $n \geqslant 0$. Since $-K_{Y}|_{E}$ is a section of the fibration
$E \longrightarrow C$, we can write
$$
-K_{Y}|_{E} \sim h + (n+a)l,
$$
where $h$ is the negative section and $l$ is a fibre. We have $a \geqslant 0$ because $-K_{Y}|_{E}$ is nef.

Thus, we obtain
$$
0 > -K_{Y'} \cdot C = K_{Y}^{2} \cdot E - 2 + 2p_{a}(C) = (-K_{Y}|_{E})^{2} - 2 = n + 2a - 2,
$$
which implies that $a = 0$ and $n \leqslant 1$.
\end{proof}

\begin{corollary}
\label{theorem:1-cor-contraction}
In the assumptions of Lemma~\ref{theorem:1-contraction},
\begin{itemize}
\item If $E \simeq \mathbb{P}^{1} \times \mathbb{P}^{1}$, then $\varphi(E)$ is a line on $X$ (a curve $\Gamma\simeq\mathbb{P}^{1}$
such that $-K_{X} \cdot \Gamma = 1$) and $X$ is singular along $\varphi(E)$;
\item If $E \simeq \mathbb{F}_{1}$, then $\varphi(E)$ is a plane on $X$ (a surface $\Pi \simeq \mathbb{P}^{2}$ such that
$K_{X}^{2} \cdot \Pi = 1$).
\end{itemize}
\end{corollary}

\begin{proof}
In the notation from the proof of Lemma~\ref{theorem:1-contraction}, if $E \simeq \mathbb{P}^{1} \times \mathbb{P}^{1}$,
then $-K_{Y} \cdot h = 0$ and $-K_{Y} \cdot l = 1$. This implies that $\varphi(E)$ is a curve $\Gamma\simeq\mathbb{P}^{1}$
on $X$ with $-K_{X} \cdot \Gamma = 1$.

If $E \simeq \mathbb{F}_{1}$, then $-K_{Y} \cdot h = 0$ and $h$ is the only curve on $E$ having zero intersection with $-K_{Y}$.
This implies that $\varphi(E)$ is a surface $\Pi \simeq \mathbb{P}^{2}$ on $X$ with
$K_{X}^{2} \cdot \Pi = K_{Y}^{2} \cdot E = 1$.
\end{proof}

Finally, recall that a normal three-dimensional singularity $o \in V$ is called a \emph{$\mathrm{cDV}$ singularity}
if it is analytically isomorphic to a hypersurface singularity $f(x,y,z) + tg(x,y,z) = 0$ in $\mathbb{C}^{4}$ with
coordinates $x$, $y$, $z$, $t$, where $f(x,y,z)=0$ is an equation of a Du Val singularity.

\begin{proposition}[\cite{Reid-canonical-threefolds}]
\label{theorem:Reid-CDV}
If $o \in V$ is a $\mathrm{cDV}$ point, then the discrepancy of every prime divisor with center at $o$ is strictly positive.
\end{proposition}

\begin{corollary}
\label{theorem:isolated-CDV-implies-terminal}
If $o \in V$ is an isolated $\mathrm{cDV}$ point, then it is terminal.
\end{corollary}

\begin{proposition}[{\cite[Corollary 2.10]{Reid-canonical-threefolds}}]
\label{theorem:inductive-CDV}
Suppose that $V$ is projective. Then
$o \in V$ is a $\mathrm{cDV}$ point if and only if for general
hyperplane section $H$ through $o$ singularity $o \in H$ is Du
Val.
\end{proposition}

\section{Beginning of the proof of Theorem ~\ref{theorem:main}: preliminary results}
\label{section:begin}

We use the notation and assumptions from Section~\ref{section:basic}. Let $X$ be a Fano threefold satisfying the conditions
of Theorem~\ref{theorem:main} and $Y$ be its terminal $\mathbb{Q}$-factorial modification.

Fix a $K$-negative extremal contraction $f: Y \longrightarrow Y'$. By Proposition~\ref{theorem:non-birational-contraction}
and \cite[(2.3.2)]{Mori-flip} morphism $f$ is birational with exceptional divisor $E$.

Throughout this section we will assume that $Y'$ has terminal factorial singularities and $-K_{Y'}$ is nef.
According to Lemmas~\ref{theorem:0-contraction} and \ref{theorem:1-contraction} this in particular implies that $Y'$ is a
weak Fano threefold.

By Remark~\ref{remark:K-trivial-contraction-1} $Y'$ is a terminal $\mathbb{Q}$-factorial modification
of some Fano threefold $X'$ with canonical Gorenstein singularities. Let us denote by $\psi: Y' \longrightarrow X'$ the
corresponding anti-canonical morphism.

\begin{lemma}
\label{theorem:contraction-to-point}
In the above notation, if $f(E)$ is a point and $-K_{Y'}^{3}=72$, then $-K_{Y}^{3}\notin\{66,68,70\}$.
\end{lemma}

\begin{proof}
By Theorem~\ref{theorem:Prokhorov-degree} threefold $Y'$ is a
terminal $\mathbb{Q}$-factorial modification either of $\mathbb{P}(3,1,1,1)$ or of $\mathbb{P}(6,4,1,1)$.
Then by Remarks~\ref{remark:unique-terminal-modification-1} and \ref{remark:unique-terminal-modification-2} $Y'$
is non-singular.

According to \cite{Cutkosky} morphism $f$ is the blow up of the maximal ideal of the point $f(E) \in Y'$. Since $Y'$ is non-singular, this
in particular implies that
$K_{Y}=f^{*}K_{Y'}+2E$ and $E^{3}=1$. Then  we get
$$
-K_{Y}^{3}=72-8E^{3}=64.
$$
\end{proof}

\begin{proposition}
\label{theorem:contraction-to-curve} In the above notation, if
$f(E)$ is a curve and $-K_{Y'}^{3}=72$, then $-K_{Y}^{3}=70$ and
$X$ is the image under birational projection of the
anti-canonically embedded Fano threefold $\mathbb{P}(6,4,1,1)
\subset \mathbb{P}^{38}$ from a singular point on
$\mathbb{P}(6,4,1,1)$ of type $\mathrm{cDV}$.
\end{proposition}

\begin{proof}
By Theorem~\ref{theorem:Prokhorov-degree} threefold $Y'$ is a terminal
$\mathbb{Q}$-factorial modification either of $\mathbb{P}(3,1,1,1)$ or of $\mathbb{P}(6,4,1,1)$.
In particular, either $X' = \mathbb{P}(3,1,1,1)$ or $\mathbb{P}(6,4,1,1)$ (see Remark~\ref{remark:K-trivial-contraction-2}).

By Remarks~\ref{remark:unique-terminal-modification-1} and  \ref{remark:unique-terminal-modification-2} threefold $Y'$ is
isomorphic to that constructed either in Example~\ref{example:examp-1} or in Example~\ref{example:examp-2}.

Set $C:=f(E)$. According to \cite{Cutkosky} morphism $f$ is the blow up of $C$.

\begin{lemma}
\label{theorem:simple-case}
In the above notation, we have $X' \ne \mathbb{P}(3,1,1,1)$.
\end{lemma}

\begin{proof}
Suppose that $X'=\mathbb{P}(3,1,1,1)$. According to Example~\ref{example:examp-1} the $\psi$-exceptional locus $E_{\psi}$
is an irreducible divisor which is contracted to the singular point on $X'$.

We have $E_{\psi} \cap C =\emptyset$. Indeed, otherwise, since $\psi(E_{\psi})$ is a point, we can find a curve
$Z \subset Y$ such that $K_{Y'} \cdot f_{*}Z=0$ and $E \cdot Z > 0$. Then from the equality
$$
K_{Y}=f^{*}K_{Y'}+E
$$
we get $K_{Y} \cdot Z > 0$. This is impossible because $-K_{Y}$ is nef.

Let us consider the following commutative diagram:

$$
\xymatrix{
Y \ar@{->}[d]_{\varphi}\ar@{->}[r]^{f}&Y'\ar@{->}[d]^{\psi}\\
X&\ar@{-->}[l]_{f'}X'}
$$

Morphism $\varphi$ is given by the linear system
$|-K_{Y}|=|f^{*}(-K_{Y'})-E|$. Hence the map $\varphi \circ
f^{-1}$ is given by the linear system  $|-K_{Y'}-C|$. Since
$E_{\psi} \cap C =\emptyset$, the map $f'$ is given by the linear
system $|-K_{X'}-\psi(C)|$. According to
Propositions~\ref{theorem:non-free-antican-system} and
\ref{theorem:free-antican-system-1} threefold $X'$ is
anti-canonically embedded in $\mathbb{P}^{38}$. This implies that
the curve $\psi(C)$ is cut out on $X'$ by some linear subspace $V
\subset \mathbb{P}^{38}$.

On the other hand, $X = X_{2g-2} \subset \mathbb{P}^{g+1}$ where $g \in \{34, 35, 36\}$ (see Remark~\ref{remark:anti-canonical-embedding}).
Thus, $\dim V \le 2$ and
$-K_{X'} \cdot \psi(C) \leqslant 2$ because $X'$ is an intersection of quadrics
(see Proposition~\ref{theorem:free-antican-system-2}).

For $X'=\mathbb{P}(3,1,1,1)$ we have $-K_{X'}\sim\mathcal{O}_{X'}(6)$ (see \cite{Iano-Fletcher}). Hence
$$
0 < \mathcal{O}_{X'}(1)\cdot \psi(C) \leqslant \frac{1}{3}.
$$
This implies that the curve $\psi(C)$ passes through the singular point on $X'$ which is impossible because
$E_{\psi} \cap C = \emptyset$.
\end{proof}

Thus, $Y'$ is a terminal $\mathbb{Q}$-factorial modification of $X' = \mathbb{P}(6,4,1,1)$. Let $P$ and $Q$ be two
non-$\mathrm{cDV}$ points on $X'$ introduced in Example~\ref{example:examp-2}.

\begin{lemma}
\label{theorem:exceptionality}
In the above notation, $\psi(C)$ is a point distinct from $P$ and $Q$.
\end{lemma}

\begin{proof}
Suppose that $\psi(C)$ is a curve. Then for the $\psi$-exceptional locus $E_{\psi}$
we have $E_{\psi} \cap C = \emptyset$. Indeed, otherwise we can find a curve $Z \subset Y$ such that $K_{Y'} \cdot f_{*}Z=0$ and $E \cdot Z > 0$. Then from the equality $$K_{Y}=f^{*}K_{Y'}+E$$ we get $K_{Y} \cdot Z > 0$. This is impossible because $-K_{Y}$ is nef.

Now repeating word by word the arguments from the proof of Lemma~\ref{theorem:simple-case} we  obtain that $-K_{X'} \cdot \psi(C) \leqslant 2$.
For $X'=\mathbb{P}(6,4,1,1)$ we have $-K_{X'}\sim\mathcal{O}_{X'}(12)$ (see \cite{Iano-Fletcher}).
Hence
$$
0 < \mathcal{O}_{X'}(1)\cdot \psi(C) \leqslant \frac{1}{6}.
$$
This implies that the curve $\psi(C)$ passes through the singular point on $X'$ which is impossible because $E_{\psi} \cap C = \emptyset$. The obtained contradiction shows that $\psi(C)$ is a point.

Now if $\psi(C)$ is either $P$ or $Q$, then we can find a curve $Z \subset Y$ such that $K_{Y'} \cdot f_{*}Z=0$ and $E \cdot Z > 0$. From the equality $$K_{Y}=f^{*}K_{Y'}+E$$
we get $K_{Y} \cdot Z > 0$. This is impossible because $-K_{Y}$ is nef.
\end{proof}

\begin{remark}
\label{remark:smooth-rational-curve}
Set $O:=\psi(C)$. According to Example~\ref{example:examp-2} $O \in X'$ is a point of type $\mathrm{cA_{1}}$. Then
Lemma~\ref{theorem:exceptionality} in particular implies that $C = \psi^{-1}(O)$ is an irreducible rational curve.
\end{remark}

According to Propositions~\ref{theorem:non-free-antican-system}
and
\ref{theorem:free-antican-system-1} threefold $X'$ is anti-canonically embedded in $\mathbb{P}^{38}$. In what follows, we consider $X'$ with respect to
this embedding.
In the notation of Remark~\ref{remark:smooth-rational-curve}, let us denote by $\pi$ the projection of the threefold $X'$ from the point $O$.

\begin{lemma}
\label{theorem:non-simple-case}
$\pi(X')$ is a Fano threefold with canonical Gorenstein  singularities and degree   $70$.
\end{lemma}

\begin{proof}
Let $\sigma: W \longrightarrow X'$ be the blow up of the maximal ideal of the point $O$. We get the following
commutative diagram:

$$
\xymatrix{
&&W\ar@{->}[ld]_{\sigma}\ar@{->}[rd]^{\tau}&&\\%
&X'\ar@{-->}[rr]_{\pi}&&\pi(X')&}
$$

The projection $\pi$ is given by the linear system $\mathcal{H} \subset |-K_{X'}|$ of all hyperplane sections of $X'$
passing through $O$.
Since $O \in X'$ is a $\mathrm{cA_{1}}$ point (see Remark~\ref{remark:smooth-rational-curve}),
for general $H \in \mathcal{H}$ we have
$$
\sigma_{*}^{-1}H=\sigma^{*}(H)-E_{\sigma},
$$
where $E_{\sigma}$ is the $\sigma$-exceptional divisor. On the other hand, by the adjunction formula we have
$$
K_{W}=\sigma^{*}K_{X'}+E_{\sigma}.
$$

Thus, morphism $\tau: W \longrightarrow \pi(X')$ is given by the linear system $\sigma_{*}^{-1}\mathcal{H} \subseteq |-K_{W}|$. In particular, $W$ is a weak Fano threefold and $\pi(X')$ is a Fano threefold. Moreover, by construction singularities of $W$ are canonical and Gorenstein.
Then, since morphism $\tau$ is crepant, by \cite{Kawamata} singularities of $\pi(X')$ are also
canonical and Gorenstein.

Finally, since $-K_{X'}^{3}=72$ and $\mathrm{mult}_{O}(X')=2$, the degree of $\pi(X')$
equals $70$.
\end{proof}

\begin{lemma}
\label{theorem:isomorphism-to-projection}
In the above notation, we have $X=\pi(X')$.
\end{lemma}

\begin{proof}
We use the notation from the proof of Lemma~\ref{theorem:non-simple-case}.
Let $L \simeq \mathbb{P}^{1}$ be the singular locus of $X'$ (see Example~\ref{example:examp-2}).
The morphism $\psi$ factors through the blow up of the curve $L$.
Then by Lemma~\ref{theorem:exceptionality}, Remark~\ref{remark:smooth-rational-curve} and the property of blow ups we
obtain the following commutative diagram:

$$
\xymatrix{
Y \ar@{->}[d]_{\eta}\ar@{->}[r]^{f}&Y'\ar@{->}[d]^{\psi}\\
W\ar@{->}[r]_{\sigma}&X'}
$$
where $\eta$ is the contraction of the proper transform of the $\psi$-exceptional divisor $E_{\psi}$. In particular, we have
$K_{Y} = \eta^{*}K_{W}$.

Since $\pi(X')$ is the anti-canonical image of $W$, this implies that
$Y$
is a terminal $\mathbb{Q}$-factorial modification of $\pi(X')$. Hence $X = \pi(X')$ by
Remark~\ref{remark:K-trivial-contraction-2}.
\end{proof}

Proposition~\ref{theorem:contraction-to-curve} is completely proved.
\end{proof}

\begin{proposition}
\label{theorem:singularities-of-X} In the assumptions of
Proposition~\ref{theorem:contraction-to-curve}, threefold $X$ has
only one singular point which is non-$\mathrm{cDV}$.
\end{proposition}

\begin{proof}
In the notation of Example~\ref{example:examp-2}, let $L$ be the singular locus of $\mathbb{P} = \mathbb{P}(6,4,1,1)$.

\begin{lemma}
\label{theorem:line}
In the above notation, the curve $L$ is a line (i.e. $L \simeq \mathbb{P}^{1}$ and $-K_{\mathbb{P}} \cdot L = 1$).
\end{lemma}

\begin{proof}
The curve $L$ is given by the equations $x_{2}=x_{3}=0$ where
$x_{0}$, $x_{1}$, $x_{2}$, $x_{3}$ are weighted projective
coordinates on $\mathbb{P}$ of weights $6$, $4$, $1$, $1$,
respectively (see \cite[5.15]{Iano-Fletcher}). This in particular
implies that $L \simeq \mathbb{P}^{1}$. It remains to show that
$-K_{\mathbb{P}} \cdot L = 1$.

Let $S$ be the surface in $\mathbb{P}$ with equation $x_{3}=0$.
Then $L \subset S$ and
$$
-K_{\mathbb{P}} \cdot  L=(-K_{\mathbb{P}}|_{S} \cdot L).
$$
The last intersection is taken on  $S = \mathbb{P}(6,4,1) \simeq
\mathbb{P}(3,2,1)$ (see \cite[5.7]{Iano-Fletcher}).

Since $-K_{\mathbb{P}}\sim\mathcal{O}_{\mathbb{P}}(12)$,
general element $D \in |-K_{\mathbb{P}}|$ is given by the equation
$$
a_{0}x_{0}^{2}+a_{6}(x_{2}, x_{3})x_{0}+a_{2}(x_{2}, x_{3})x_{0}x_{1}+a_{4}(x_{2}, x_{3})x_{1}^{2}+a_{8}(x_{2}, x_{3})x_{1}+
a_{12}(x_{2}, x_{3})=0,
$$
where $a_{i}(x_{2}, x_{3})$ are homogeneous polynomials of degree $i$ on the variables $x_{2}$, $x_{3}$.

If $y_{0}$, $y_{1}$, $y_{2}$ are weighted projective coordinates on
$S$ of weights $3$, $2$, $1$, respectively, then $D|_{S}$ is given
by the equation (see \cite[5.10]{Iano-Fletcher})
$$b_{0}y_{0}^{2}+b_{1}y_{0}y_{3}^{3}+
b_{2}y_{0}y_{1}y_{2}+ b_{3}y_{1}^{3}+b_{4}y_{1}^{2}y_{2}^{2}+b_{5}y_{1}y_{2}^{4}+b_{6}y_{2}^{6}=0
$$
on $S$, where $b_{i}\in\mathbb{C}$. This implies that $D|_{S} \sim \mathcal{O}_{S}(6)$. Thus, we have
$$
-K_{\mathbb{P}}\cdot L = (-K_{\mathbb{P}}|_{S} \cdot L) = (\mathcal{O}_{S}(6) \cdot \mathcal{O}_{S}(1)) =1
$$
for the last intersection on $S=\mathbb{P}(3,2,1)$ (see \cite{Dolgachev}).
\end{proof}

Let $P$ and $Q$ be two non-$\mathrm{cDV}$ points on $\mathbb{P} = \mathbb{P}(6,4,1,1)$ introduced in Example~\ref{example:examp-2}.
\begin{lemma}
\label{theorem:unique-line}
In the above notation, $L$ is the unique line on $\mathbb{P}$ passing through singular points on
$\mathbb{P}$ distinct from $P$ and $Q$.
\end{lemma}

\begin{proof}
Let $L_{0} \ne L$ be a line on $\mathbb{P}$ intersecting the line $L$ at some point other than $P$ and $Q$.

In the notation from the proof of Lemma~\ref{theorem:line}, let
$\tau_{\lambda_{1}, \lambda_{2}}$,
$\lambda_{i}\in\mathbb{C}\setminus\{0\}$, be the automorphism of
$\mathbb{P}$ given by the following formula:
$$
\tau_{\lambda_{1}, \lambda_{2}}: [x_{0}:x_{1}:x_{2}:x_{3}] \mapsto [\lambda_{1}x_{0}:\lambda_{2}x_{1}:x_{2}:x_{3}].
$$
The group generated by $\tau_{\lambda_{1}, \lambda_{2}}$ acts transitively
on $L \setminus \{P, Q\}$ (see \cite{Iano-Fletcher}, 5.15).

Apply
automorphisms $\tau_{\lambda_{1}, \lambda_{2}}$ to $L_{0}$.
We obtain a family of lines $\{L_{t}\}$, $t \in L\setminus\{P, Q\}$, on $\mathbb{P}$ such that
each line $L_{t}$ intersects $L$ at some point $P_{t}$  with $P_{t} \ne P_{t'}$ for $t \ne t'$.

Let $T$ be the terminal $\mathbb{Q}$-factorial modification of $\mathbb{P}$ and
$\phi: T \longrightarrow \mathbb{P}$ the corresponding anti-canonical morphism. Let $L'_{t}$ be the proper transform of line $L_{t}$ on $T$. We have $-K_{T} \cdot L'_{t} = 1$.
Moreover, set $E_{P}:=\phi^{-1}(P)$, $E_{Q}:=\phi^{-1}(Q)$ and $E_{L}:=\phi^{-1}(L)$. Then
$$
E_{P}\cdot L'_{t}=E_{Q}\cdot L'_{t}=0\qquad\mbox{and}\qquad
E_{L} \cdot L'_{t}=1
$$
on $T$. Since for the $\phi$-exceptional locus $E_{\phi}$ we have $E_{\phi} = E_{P}\cup E_{Q}\cup E_{L}$ (see Example~\ref{example:examp-2}),
the obtained equalities imply that the curves in the family $\{L'_{t}\}$ are numerically equivalent.

Furthermore, from the equality $-K_{T} \cdot L'_{t} = 1$ and Proposition~\ref{theorem:extremal-rays} we deduce that the ray
$R: = \mathbb{R}_{+}[L'_{t}]$ determines an extremal contraction $\varphi_{R}: T \longrightarrow T'$.
By Propositions 4.11 and 5.2 in \cite{Prokhorov-degree} and \cite[(2.3.2)]{Mori-flip}
the contraction $\varphi_{R}$ is birational and contracts the surface
$$
E_{R}: = \bigcup_{t \in L\setminus\{P, Q\}} L'_{t}
$$
to some curve.

If $T'$ is a weak Fano threefold, then by Lemma~\ref{theorem:1-contraction} we have
$-K_{T'}^{3} \geqslant -K_{T}^{3} = 72$. Thus, by Remark~\ref{remark:K-trivial-contraction-1} and
Theorem~\ref{theorem:Prokhorov-degree} threefold $T'$ is a terminal $\mathbb{Q}$-factorial modification either of
$\mathbb{P}(3,1,1,1)$ or of $\mathbb{P}(6,4,1,1)$.
By Remarks~\ref{remark:unique-terminal-modification-1} and \ref{remark:unique-terminal-modification-2} either
$T' \simeq T$ or $T'$ is that constructed in  Example~\ref{example:examp-1}. This implies that
$\rho(T') = \rho(T) - 1 = 4$ equals either $5$ or $2$, a contradiction.

Hence by Lemma~\ref{theorem:1-contraction} divisor $-K_{T'}$ is not nef and either
$E_{R} = \mathbb{F}_{1}$ or $\mathbb{P}^{1} \times \mathbb{P}^{1}$.
But according to
Corollary~\ref{theorem:1-cor-contraction} equality
$E_{R} = \mathbb{P}^{1} \times \mathbb{P}^{1}$ implies that $E_{R} \subset E_{\phi}$
which is not the case.

Thus, we get $E_{R} = \mathbb{F}_{1}$. But $E_{P} \cap E_{R}$, $E_{Q} \cap E_{R} \ne \emptyset$ and $\phi(E_{P})$, $\phi(E_{Q})$
are two distinct points on $\mathbb{P}$. This implies that there
are two distinct curves on $E_{R}$ having zero intersection with $K_{T}$.
But according to calculations in the proof of Lemma~\ref{theorem:1-contraction} the negative section on
$E_{R} = \mathbb{F}_{1}$ is the only curve on $E_{R}$ having this property.
The obtained contradiction completes the proof of Lemma~\ref{theorem:unique-line}.
\end{proof}

We use the notation from the proof of
Proposition~\ref{theorem:contraction-to-curve}. By
Lemmas~\ref{theorem:line} and \ref{theorem:unique-line} morphism
$\tau: W \longrightarrow X=\pi(X')$ contracts only the proper
transform $L'$ of the line $L$ to the unique singular point on
$X$. Set $o:=\tau(L')$.

\begin{lemma}
\label{theorem:non-CDV}
Singularity $o \in X$ is non-$\mathrm{cDV}$.
\end{lemma}

\begin{proof}
If the point $o \in X$ is $\mathrm{cDV}$, then $X$ has terminal Gorenstein singularities by Corollary~\ref{theorem:isolated-CDV-implies-terminal}. But this contradicts Theorem~\ref{theorem:Namikawa-smoothing} because $-K_{X}^{3}=70$.
\end{proof}

Proposition~\ref{theorem:singularities-of-X} is completely proved.
\end{proof}

\begin{remark}
\label{remark:rank-of-Pic}
From Proposition~\ref{theorem:singularities-of-X} it is not difficult to see that $\rho(X)=1$ for Fano threefold $X$ satisfying the conditions of
Proposition~\ref{theorem:contraction-to-curve}. We will denote this $X$ by $X_{70}$. Furthermore, in the
notation from the proof of Lemma~\ref{theorem:unique-line}, since the group of automorphisms $\tau_{\lambda_{1}, \lambda_{2}}$
acts transitively on $L \setminus \{P, Q\}$, from Lemmas~\ref{theorem:simple-case}--\ref{theorem:isomorphism-to-projection}
we deduce that the threefold $X_{70}$ is unique up to isomorphism.
\end{remark}

\begin{lemma}
\label{theorem:smooth-terminal-modification-C}
In the above notation, if $X = X_{70}$, then every terminal $\mathbb{Q}$-factorial modification of
$X$ is non-singular.
\end{lemma}

\begin{proof}
We use the notation from the proof of Proposition~\ref{theorem:contraction-to-curve}. We have shown that the morphism
$f: Y \longrightarrow Y'$ is the blow up of a non-singular curve on a non-singular threefold $Y'$. Thus, $X$ possess a non-singular terminal $\mathbb{Q}$-factorial modification.

Now from Theorem 2.4 in \cite{Kollar-flops} and Remark~\ref{remark:terminal-modifications-are-connected-by-flops} we
deduce that every terminal $\mathbb{Q}$-factorial modification of $X$ is non-singular.
\end{proof}

\begin{corollary}
\label{theorem:contraction-to-point-C}
In the above notation, if $f(E)$ is a point and $X' = X_{70}$, then $-K_{Y}^{3}\notin\{66,68,70\}$.
\end{corollary}

\begin{proof}
By Lemma~\ref{theorem:smooth-terminal-modification-C} threefold $Y'$ is non-singular. Then
repeating word by word the arguments from the proof of Lemma~\ref{theorem:contraction-to-point} we obtain that
$-K_{Y}^{3}=62$.
\end{proof}

\begin{corollary}
\label{theorem:contraction-to-curve-C}
In the above notation, if $f(E)$ is a curve and $X' = X_{70}$, then $-K_{Y}^{3}\notin\{66,68,70\}$.
\end{corollary}

\begin{proof}
Set $C:=f(E)$. By Lemma~\ref{theorem:smooth-terminal-modification-C} threefold $Y'$ is non-singular and according to
\cite{Cutkosky} $f$ is the blow up of $C$.

If $\psi(C)$ does not pass
through the singular point on $X'$, then repeating word by word the arguments from the proof of
Lemma~\ref{theorem:simple-case} we get
$$
-K_{X'} \cdot \psi(C) \leqslant 2.
$$
By construction of $X_{70}$ and our assumption $X'$ is isomorphic
to $\mathbb{P}=\mathbb{P}(6,4,1,1)$ near $\psi(C)$. In particular,
the image $C'$ of the curve $\psi(C)$ on $\mathbb{P}$ does not
pass through the singular points.

On the other hand, we have

$$
0 < \mathcal{O}_{\mathbb{P}}(1)\cdot C' = \frac{1}{12}(-K_{X'} \cdot \psi(C))\leqslant \frac{1}{6}.
$$
This implies that the curve $C'$ passes through the singular point on $\mathbb{P}$, a contradiction.

Thus, we get $E_{\psi} \cap C \ne \emptyset$ for the $\psi$-exceptional locus $E_{\psi}$.
But $\psi(E_{\psi})$ is a point on $X'$. Hence we can find a curve $Z \subset Y$ such that $K_{Y'} \cdot f_{*}Z = 0$
and $E \cdot Z > 0$. Then from the equality
$$
K_{Y}=f^{*}K_{Y'}+E
$$
we get $K_{Y} \cdot Z > 0$. This is impossible because $-K_{Y}$ is nef.
\end{proof}

\begin{remark}
\label{remark:no-nef-contractions} It follows from
Lemma~\ref{theorem:contraction-to-point},
Propositions~\ref{theorem:contraction-to-curve} and
\ref{theorem:singularities-of-X},
Corollaries~\ref{theorem:contraction-to-point-C} and
\ref{theorem:contraction-to-curve-C}, the estimates $\rho(Y') <
\rho(Y)$ and $-K_{Y'}^{3} \geqslant -K_{Y}^{3}$ (see
Lemmas~\ref{theorem:0-contraction} and
\ref{theorem:1-contraction}) that to complete the proof of
Theorem~\ref{theorem:main} it remains to consider the case when
there are no $K$-negative extremal contractions of $Y$ on weak
Fano threefolds with terminal factorial singularities. Then by
Corollaries~\ref{theorem:0-cor-contraction} and
\ref{theorem:1-cor-contraction} either the corresponding Fano
threefold $X$ is singular along a line or contains a plane.
\end{remark}

\section{Reduction to the log Mori fibration}
\label{section:reduction}

In this section we follow $\S 6$ in \cite{Prokhorov-degree}. Let $X$ be a Fano threefold satisfying the conditions
of Theorem~\ref{theorem:main}. Recall that by Remark~\ref{remark:anti-canonical-embedding} threefold
$X = X_{2g-2}$ is anti-canonically embedded in $\mathbb{P}^{g+1}$ with $g \in \{34,35,36\}$ and $X_{2g-2}$ is an intersection
of quadrics.

According to Remark~\ref{remark:no-nef-contractions} in order to complete the proof of
Theorem~\ref{theorem:main} we may assume that either $X$ is singular along a line $\Gamma$ or contains a plane $\Pi$
(cases ${\bf A}$ and ${\bf B}$, respectively). It is also convenient to distinguish the case ${\bf C}$: threefold $X$ has at
least one non-$\mathrm{cDV}$ point $O$.

Set $\mathcal{L}: = |-K_{X}|$ and consider the following linear
systems $\mathcal{H} \subset \mathcal{L}$:
$$
\begin{array}{c}
\mathcal{H}:=\{H\in \mathcal{L}|H\supset \Gamma\}\qquad\mbox{in case ${\bf A}$;}\\
\\
\mathcal{H}:=\{H|H+\Pi\in\mathcal{L}\}\qquad\mbox{in case ${\bf B}$;}\\
\\
\mathcal{H}:=\{H\in \mathcal{L}|H \ni O\}\qquad\mbox{in case ${\bf C}$.}\\
\end{array}
$$

Take a terminal $\mathbb{Q}$-factorial modification $\varphi: Y \longrightarrow X$ of $X$. Let $\mathcal{L}_{Y}$ and
$\mathcal{H}_{Y}$ be the proper transforms on $Y$ of linear systems $\mathcal{L}$ and $\mathcal{H}$, respectively. We
have
$$
\varphi^{*}\mathcal{H} = \mathcal{H}_{Y}+D_{Y}
$$
where $D_{Y} \ne 0$ is an integral effective exceptional divisor.

\begin{remark}
\label{remark:birationally-rulled-components}
According to Corollary 2.14 in \cite{Reid-canonical-threefolds} all the components of $D_{Y}$ are surfaces of negative Kodaira dimension.
\end{remark}

Furthermore, we have
$$
\begin{array}{c}
K_{Y}+\mathcal{H}_{Y}+D_{Y}=\varphi^{*}(K_{X}+\mathcal{H})\sim 0\qquad\mbox{in cases ${\bf A}$ and ${\bf C}$;}\\
\\
K_{Y}+\mathcal{H}_{Y}+D_{Y}=\varphi^{*}(K_{X}+\mathcal{H}+\Pi)\sim 0\qquad\mbox{in case ${\bf B}$.}
\end{array}
$$

\begin{lemma}[{\cite[Lemma 6.5]{Prokhorov-degree}}]
\label{theorem:image-of-projection} In the above notation, the image
of the threefold $X$ under the map $\varphi_{\mathcal{H}}$ given by
the linear system $\mathcal{H}$ is of dimension $3$.
\end{lemma}

\begin{lemma}[{\cite[Lemma 6.8]{Prokhorov-degree}}]
\label{theorem:good-properties} In the above notation, one can take
$Y$ so that
\begin{itemize}
\item the pair $(Y, \mathcal{H}_{Y})$ is canonical;
\item the linear system $\mathcal{H}_{Y}$ is nef and consists of Cartier divisors.
\end{itemize}
\end{lemma}

To the pair $(Y, \mathcal{H}_{Y})$ constructed in
Lemma~\ref{theorem:good-properties} let us apply
$(K_{Y}+\mathcal{H}_{Y})$-Minimal Model Program. Each step
preserves the equality $K + \mathcal{H} \equiv -D$ for some
divisor $D>0$. Hence at the end we obtain a pair $(W,
\mathcal{H}_{W})$ with $(K_{W}+\mathcal{H}_{W})$-negative
contraction $g: W \longrightarrow V$ to a lower-dimensional
variety $V$.

\begin{remark}
\label{remark:singularities-of-pairs} By \cite[Lemma
3.4]{Prokhorov-degree} $(K_{Y}+\mathcal{H}_{Y})$-Minimal Model
Program preserves the property of $\mathcal{H}_{Y}$ being nef and
consisting of Cartier divisors. This implies that $W$ has only
terminal $\mathbb{Q}$-factorial singularities. Moreover, let us
denote by $\mathcal{L}_{W}$ the proper transform on $W$ of the
linear system $\mathcal{L}_{Y}$. Then $K_{W} + \mathcal{L}_{W}
\sim 0$ and by Lemma 3.1 in \cite{Prokhorov-degree} the pair $(W,
\mathcal{L}_{W})$ is canonical. In particular, the  anti-canonical
linear system $|-K_{W}|$ does not have fixed components.
\end{remark}

\begin{remark}
\label{remark:linear-systems}
By construction the initial Fano threefold $X$ is the image of $W$ under the birational map
$\varphi_{\mathcal{L}_{W}}$ given by the linear system $\mathcal{L}_{W} \subset |-K_{W}|$. Furthermore, from Lemma~\ref{theorem:image-of-projection} we deduce that the linear system $\mathcal{H}_{W}$ is ample over $V$ and
does not have fixed components.
\end{remark}

We have
$$
K_{W}+\mathcal{H}_{W}+D_{W}\sim 0,\qquad D_{W} > 0
$$
on $W$ and the inequalities
$$
\begin{array}{c}
\dim|-K_{W}| \geqslant \dim\mathcal{L}_{W} = \dim|-K_{X}|,\\
\\
\dim|H| \geqslant \dim\mathcal{H},
\end{array}
$$
where $H \in \mathcal{H}_{W}$ is a general divisor.

\begin{remark}
\label{remark:inequalities} By Remark~\ref{remark:linear-systems}
we have $\dim |-K_{W}| \geqslant \dim |-K_{X}|$. By construction
of the linear system $\mathcal{H}$ we have $\dim\mathcal{H}
\geqslant \dim |-K_{X}| - 3$. Then from the assumption $-K_{X}^{3}
\in \{66,68,70\}$ we obtain that $\dim |-K_{W}| \geqslant 35$ and
$\dim |H| \geqslant 32$.
\end{remark}

By construction variety $V$ is a point, a curve or a surface. Moreover, when $V$ is a point by Proposition 7.2 in \cite{Prokhorov-degree}
we have $\dim |-K_{W}| \leqslant 34$ which contradicts the estimate from Remark~\ref{remark:inequalities}.

In what follows we will treat separately the cases when $V$ is a curve and a surface.

\section{Proof of Theorem~\ref{theorem:main} in the case when $V$ is a curve}
\label{section:curve}

We use the notation and assumptions from
Section~\ref{section:reduction}. From the Leray spectral sequence
and Kawamata--Vieweg Vanishing Theorem we deduce that $H^{1}(V,
\mathcal{O}_{V})=H^{1}(W, \mathcal{O}_{W})=0$. Thus, we have $V
\simeq \mathbb{P}^{1}$ in the present case.

General fibre $W_{\eta}$ of the contraction $g: W \longrightarrow V$ is a non-singular del Pezzo
surface.
For $H \in \mathcal{H}_{W}$ divisor $-(K_{W_{\eta}}+H|_{W_{\eta}})$ is ample by construction and divisor $H|_{W_{\eta}}$ is
ample by Remark~\ref{remark:linear-systems}. This implies that $W_{\eta} \simeq \mathbb{P}^{2}$ or
$\mathbb{P}^{1}\times\mathbb{P}^{1}$. Moreover, for $W_{\eta} \simeq \mathbb{P}^{2}$ we have either
$H|_{W_{\eta}} \simeq \mathcal{O}_{\mathbb{P}^{2}}(1)$ or $\mathcal{O}_{\mathbb{P}^{2}}(2)$.
\begin{lemma}
\label{theorem:special-scrolls}
In the above notation, if $W_{\eta} \simeq \mathbb{P}^{2}$ and $H|_{W_{\eta}} \simeq \mathcal{O}_{\mathbb{P}^{2}}(1)$,
then $W$ is a $\mathbb{P}^{2}$-bundle. Moreover, either
$W = \mathrm{Proj}\left(\mathcal{O}_{\mathbb{P}^{1}}(5)\oplus\mathcal{O}_{\mathbb{P}^{1}}(2)\oplus\mathcal{O}_{\mathbb{P}^{1}}\right)$
or $\mathrm{Proj}\left(\mathcal{O}_{\mathbb{P}^{1}}(6)\oplus\mathcal{O}_{\mathbb{P}^{1}}(2)\oplus\mathcal{O}_{\mathbb{P}^{1}}\right)$.
\end{lemma}

\begin{proof}
It follows at once from Lemma 8.1 and Proposition 8.2 in \cite{Prokhorov-degree} and Remark~\ref{remark:singularities-of-pairs} that
$W = \mathrm{Proj}\left(\mathcal{O}_{\mathbb{P}^{1}}(d_{1})\oplus\mathcal{O}_{\mathbb{P}^{1}}(d_{2})\oplus\mathcal{O}_{\mathbb{P}^{1}}\right)$
with the following possibilities for $(d_{1}, d_{2})$:
$$
(1,1),\hspace{15pt} (2,1),\hspace{15pt} (2,2),\hspace{15pt} (3,1),\hspace{15pt} (3,2),\hspace{15pt} (4,1),\hspace{15pt} (4,2),\hspace{15pt} (5,2),\hspace{15pt}
(6,2).
$$
Except for the last two cases we have $\dim |-K_{W}| \leqslant 33$ (see for example \cite[2.5]{Reid-chapters-on-surfaces}).
Hence by Remark~\ref{remark:inequalities} the pair $(d_{1},d_{2})$ is either $(5,2)$ or $(6,2)$.
\end{proof}

\begin{proposition}
\label{theorem:T-variety} In the assumptions of
Lemma~\ref{theorem:special-scrolls}, if $W =
\mathrm{Proj}\left(\mathcal{O}_{\mathbb{P}^{1}}(5)\oplus\mathcal{O}_{\mathbb{P}^{1}}(2)\oplus\mathcal{O}_{\mathbb{P}^{1}}\right)$,
then $-K_{X}^{3} = 66$ and $X$ the anti-canonical image of $W$.
\end{proposition}

\begin{proof}
For $W = \mathrm{Proj}\left(\mathcal{O}_{\mathbb{P}^{1}}(5)\oplus\mathcal{O}_{\mathbb{P}^{1}}(2)\oplus\mathcal{O}_{\mathbb{P}^{1}}\right)$ we have $\dim |-K_{W}| = 35$.
From Remark~\ref{remark:inequalities} we deduce the equality
$|-K_{W}| = \mathcal{L}_{W}$. In particular, for the initial Fano threefold $X$ we must have $-K_{X}^{3} = 66$.
Let us show that this possibility is really obtained.

On $W$ we have $-K_{W} \sim 3M - 5L$ where $M$ is the tautological line bundle and $L$ is the fibre of the morphism $g$ (see
for example \cite[A.10]{Reid-chapters-on-surfaces}). Then the anti-canonical linear system $|-K_{W}|$ is generated by
(see \cite[2.4]{Reid-chapters-on-surfaces})
\begin{equation}
\label{equation:basis-0}
\begin{array}{c}
g_{1}x_{2}^{3}, \hspace{15pt} x_{1}x_{3}^{2},\hspace{15pt} g_{2}x_{1}x_{2}x_{3},\hspace{15pt} g_{4}x_{1}x_{2}^{2},\hspace{15pt}
g_{5}x_{1}^{2}x_{3},\hspace{15pt} g_{7}x_{1}^{2}x_{2},\hspace{15pt} g_{10}x_{1}^{3}
\end{array}
\end{equation}
where $x_{1}$, $x_{2}$, $x_{3}$ are projective coordinates on the fibre $L \simeq \mathbb{P}^{2}$, $g_{i}$ is a
homogeneous polynomial of degree $i$ on the projective coordinates $t_{0}$, $t_{1}$ on the base
$V \simeq \mathbb{P}^{1}$.
In particular, the base locus of $|-K_{W}|$ on $W$ is a curve $C_{0} \simeq \mathbb{P}^{1}$ given by equations
$x_{1}=x_{2}=0$.

In order to resolve indeterminacies of the map $\varphi_{\mathcal{L}_{W}}: W \dashrightarrow X$ we may set $t_{2}=x_{3}=1$,
$x_{1}=x$, $x_{2}=y$ and $t_{1}=z$.

Let $\sigma_{1}: W_{1} \longrightarrow W_{0}=W$ be the blow up of the curve $C_{0}$ with exceptional divisor $E_{1}$.
Threefold $W_{1}$ is covered by two affine charts $U_{i}^{(1)}$ each isomorphic to $\mathbb{C}^{3}$.
We will again denote by $x$, $y$, $z$ the coordinates on each chart $U_{i}^{(1)}$.

The map $U_{1}^{(1)} \longrightarrow W_{0}$ is given by $$(x, y, z) \mapsto (x,xy,z).$$
Then the linear system $\mathcal{L}_{W_{1}} = \sigma_{1*}^{-1}\mathcal{L}_{W}$  is generated by
$$
g_{1}y^{3}x^{2}, \hspace{15pt}
1,\hspace{15pt} g_{2}xy,\hspace{15pt} g_{4}x^{2}y^{2},\hspace{15pt}
g_{5}x,\hspace{15pt} g_{7}x^{2}y,\hspace{15pt} g_{10}x^{2}
$$
on $U_{1}^{(1)}$ which implies that $\mathcal{L}_{W_{1}}$ is free on $U_{1}^{(1)}$.

The map $U_{2}^{(1)} \longrightarrow W_{0}$ is given by $$(x, y, z) \mapsto (xy,y,z).$$
Then the linear system $\mathcal{L}_{W_{1}} = \sigma_{1*}^{-1}\left(\mathcal{L}_{W}\right)$ is generated by
\begin{equation}
\label{equation:basis-1}
\begin{array}{c}
g_{1}y^{2}, \hspace{15pt} x,\hspace{15pt} g_{2}xy,\hspace{15pt} g_{4}xy^{2},\hspace{15pt}
g_{5}x^{2}y,\hspace{15pt} g_{7}x^{2}y^{2},\hspace{15pt} g_{10}x^{3}y^{2}
\end{array}
\end{equation}
on $U_{2}^{(1)}$.
This implies that the base locus of $\mathcal{L}_{W_{1}}$ on $W_{1}$ is an irreducible rational curve $C_{1} \subset E_{1} \cap U_{2}^{(1)}$ given by equations $x=y=0$.

In order to resolve the indeterminacies of the map
$\varphi_{\mathcal{L}_{W_{1}}}: W_{1} \dashrightarrow X$
we may set $W_{1} = U_{2}^{(1)}$ with coordinates $x$, $y$, $z$. Then \eqref{equation:basis-1} is the set of generators of
the linear system $\mathcal{L}_{W_{1}}$.

Let $\sigma_{2}: W_{2} \longrightarrow W_{1}$ be the blow up of the curve $C_{1}$ with exceptional divisor $E_{2}$.
Threefold $W_{2}$ is covered by two affine charts $U_{i}^{(2)}$ each isomorphic to $\mathbb{C}^{3}$.
We will again denote by $x$, $y$, $z$ the coordinates on each chart $U_{i}^{(2)}$.

The map $U_{1}^{(2)} \longrightarrow W_{1}$ is given by $$(x, y, z) \mapsto (x,xy,z)$$ and again the linear system
$\mathcal{L}_{W_{2}} = \sigma_{2*}^{-1}\left(\mathcal{L}_{W_{1}}\right)$ is free on $U_{1}^{(2)}$.

The map $U_{2}^{(2)} \longrightarrow W_{1}$ is given by $$(x, y,
z) \mapsto (xy,y,z).$$ Then the linear system
$\mathcal{L}_{W_{2}}=
\sigma_{2*}^{-1}\left(\mathcal{L}_{W_{1}}\right)$ is generated by
\begin{equation}
\label{equation:basis-2}
\begin{array}{c}
g_{1}y, \hspace{15pt} x,\hspace{15pt} g_{2}xy,\hspace{15pt} g_{4}xy^{2},\hspace{15pt}
g_{5}x^{2}y^{2},\hspace{15pt} g_{7}x^{2}y^{3},\hspace{15pt} g_{10}x^{3}y^{4}
\end{array}
\end{equation}
on $U_{2}^{(2)}$.
This implies that the base locus of $\mathcal{L}_{W_{2}}$ on $W_{2}$ is an irreducible rational curve $C_{2} \subset E_{2} \cap U_{2}^{(2)}$ given by equations $x=y=0$.

Finally, in order to resolve the indeterminacies of the map $\varphi_{\mathcal{L}_{W_{2}}}: W_{2} \dashrightarrow X$
we may set $W_{2} = U_{2}^{(2)}$ with coordinates $x$, $y$, $z$. Then \eqref{equation:basis-2} is the set of generators of
the linear system $\mathcal{L}_{W_{2}}$.

Let $\sigma_{3}: W_{3} \longrightarrow W_{2}$ be the blow up of the curve $C_{2}$ with exceptional divisor $E_{3}$.
Threefold $W_{3}$ is covered by two affine charts $U_{i}^{(3)}$ each isomorphic to $\mathbb{C}^{3}$.
We will again denote by $x$, $y$, $z$ the coordinates on each chart $U_{i}^{(3)}$.

The map $U_{1}^{(3)} \longrightarrow W_{2}$ is given by $$(x, y, z) \mapsto (x,xy,z)$$ and again the linear system
$\mathcal{L}_{W_{3}} = \sigma_{3*}^{-1}\left(\mathcal{L}_{W_{2}}\right)$ is free on $U_{1}^{(3)}$.

The map $U_{2}^{(3)} \longrightarrow W_{2}$ is given by $$(x, y, z) \mapsto (xy,y,z).$$
Then the linear system $\mathcal{L}_{W_{3}}$ is generated by
$$
g_{1}, \hspace{15pt} x,\hspace{15pt} g_{2}xy,\hspace{15pt} g_{4}xy^{2},\hspace{15pt}
g_{5}x^{2}y^{3},\hspace{15pt} g_{7}x^{2}y^{4},\hspace{15pt} g_{10}x^{3}y^{6}
$$
on $U_{2}^{(3)}$. This implies that the linear system
$\mathcal{L}_{W_{3}}=
\sigma_{3*}^{-1}\left(\mathcal{L}_{W_{2}}\right)$ is free on
$U_{2}^{(3)}$ and hence it is free on $W_{3}$.

\begin{lemma}
\label{theorem:weak-Fano}
In the above notation, $W_{3}$ is a non-singular weak Fano threefold such that the morphism $\varphi_{\mathcal{L}_{W_{3}}}$ is crepant and the image $X = \varphi_{\mathcal{L}_{W_{3}}}(W_{3})$ is
a Fano threefold with canonical Gorenstein singularities and degree $66$.
\end{lemma}

\begin{proof}
In the above notation, for $1 \leqslant i \leqslant 3$ we have
$$
K_{W_{i}} = \sigma_{i}^{*}K_{W_{i-1}} + E_{i}
$$
and general element in $\mathcal{L}_{W_{i-1}}$ is non-singular along the base curve $C_{i-1}$ (see \eqref{equation:basis-0}--\eqref{equation:basis-2}).
Since $\mathcal{L}_{W_{0}} = |-K_{W}|$, this implies that $\mathcal{L}_{W_{i}} \subset |-K_{W_{i}}|$ for all
$1 \leqslant i \leqslant 3$. In particular, divisor $-K_{W_{3}}$ is nef because the linear system $\mathcal{L}_{W_{3}}$ is free.

Furthermore, the image $\varphi_{|-K_{W}|}(W) = \varphi_{\mathcal{L}_{W_{3}}}(W_{3})$ is three-dimensional because $-K_{W}$ is big. This implies that divisor $-K_{W_{3}}$ is also big.

Thus, $W_{3}$ is a weak Fano threefold and morphism $\varphi_{\mathcal{L}_{W_{3}}}$ is crepant. Moreover, by construction
$W_{3}$ is non-singular. Then $X = \varphi_{\mathcal{L}_{W_{3}}}(W_{3})$ is a Fano threefold with canonical Gorenstein
singularities (see \cite{Kawamata}).

Finally, we have $X = X_{2g-2} \subset \mathbb{P}^{g+1}$ for the genus $g$ of $X$ equal $34$ because $X = \varphi_{|-K_{W}|}(W)$ and $\dim |-K_{W}| = 35$.
\end{proof}

Proposition~\ref{theorem:T-variety} is completely proved.
\end{proof}

\begin{remark}
\label{remark:T-variety} Obviously, Fano threefold $X$ from
Proposition~\ref{theorem:T-variety} is toric. From the precise
view of the fan of $X$ in the list obtained in
\cite{Kreuzer-Skarke} it is straightforward that $X$ is
$\mathbb{Q}$-factorial and $\rho(X) = 2$. We will denote this $X$
by $X_{66}$.
\end{remark}

\begin{proposition}
\label{theorem:singularities-of-T-variety}
In the above notation, if $X = X_{66}$, then singularities of  $X$ are non-$\mathrm{cDV}$.
\end{proposition}

\begin{proof}
We use the notation from the proof of Proposition~\ref{theorem:T-variety}. From Lemma~\ref{theorem:weak-Fano} we deduce
that $\varphi_{\mathcal{L}_{W_{3}}}: W_{3} \longrightarrow X$ is a terminal $\mathbb{Q}$-factorial modification of $X$. In particular, we have $\mathcal{L}_{W_{3}} = |-K_{W_{3}}|$.

Let $L_{t} \subset E_{2}$ be the fibre $\sigma_{2}^{-1}(t)$, $t\in\mathbb{P}^{1}$, and $L'_{t}$ be its proper transform on $W_{3}$.

\begin{lemma}
\label{theorem:zero-intersection}
On $W_{3}$ we have $K_{W_{3}} \cdot L'_{t} = 0$.
\end{lemma}

\begin{proof}
Set $E_{1}^{*}: = (\sigma_{2} \circ\sigma_{3})^{*}(E_{1})$ and $E_{2}^{*}: = \sigma_{3}^{*}(E_{2})$. On $W_{3}$ we have
$$
K_{W_{3}} = (\sigma_{1} \circ \sigma_{2}\circ\sigma_{3})^{*}(K_{W})+E_{1}^{*}+E_{2}^{*}+E_{3}.
$$

Furthermore, on $W_{3}$ we have
$$
E_{3} \cdot L'_{t} = 1, \hspace{15pt}  E_{2}^{*} \cdot L'_{t} = E_{2} \cdot L_{t} = -1, \hspace{15pt} E_{1}^{*} \cdot L'_{t} = \sigma_{2}^{*}(E_{1}) \cdot L_{t} = 0.
$$
Thus, we get $K_{W_{3}} \cdot L'_{t} = 0$ on $W_{3}$.
\end{proof}

Let $E'_{2}$ be the proper transform of the surface $E_{2}$ on $W_{3}$.

\begin{lemma}
\label{theorem:restriction-is-contracted}
On $W_{3}$ we have $K_{W_{3}} \cdot E'_{2} \cdot E_{3} = 0$.
\end{lemma}

\begin{proof}
In the notation from the proof of Proposition~\ref{theorem:T-variety},
let $\sigma_{3}: W_{3} \longrightarrow W_{2}$ be the blow up of the curve $C_{2}$ with exceptional divisor $E_{3}$.
Without loss of generality we may set $W_{2} = U_{2}^{(2)}$ with coordinates $x$, $y$, $z$. Then
the threefold $W_{3}$ is covered by two affine charts $U_{i}^{(3)}$ each isomorphic to $\mathbb{C}^{3}$.
We will again denote by $x$, $y$, $z$ the coordinates on each chart $U_{i}^{(3)}$.

The map $U_{1}^{(3)} \longrightarrow W_{2}$ is given by
$$
(x, y, z) \mapsto (x,xy,z).
$$
Then general element $S \in\mathcal{L}_{W_{3}} = |-K_{W_{3}}| = \sigma_{3*}^{-1}\left(\mathcal{L}_{W_{2}}\right)$
is given by the equation
$$
g_{1}y + g_{0} + g_{2}xy + g_{4}x^{2}y^{2}+g_{5}x^{3}y^{2}+g_{7}x^{4}y^{3}+g_{10}x^{6}y^{4}=0
$$
on $U_{1}^{(3)}$ where $g_{0} \in \mathbb{C}$ (see \eqref{equation:basis-2}).
On the other hand, the equation of $E_{3}$ on $U_{1}^{(3)}$ is $x=0$ and the equation of $E_{2}$ on $W_{2}$ is $y=0$. This
implies that
$$
S \cap E'_{2} \cap E_{3} = \emptyset
$$
on $U_{1}^{(3)}$.

Furthermore, the map $U_{2}^{(3)} \longrightarrow W_{2}$ is given by
$$
(x, y, z) \mapsto (xy,y,z)
$$
Then $E'_{2} \cap U_{2}^{(3)} \cap E_{3} = \emptyset$ because the equation of $E_{3}$ on $U_{2}^{(3)}$ is $y=0$ and the equation
of $E_{2}$ on $W_{2}$ is $y=0$. This implies that for general element $S \in\mathcal{L}_{W_{3}} = |-K_{W_{3}}|$
$$
S \cap E'_{2} \cap E_{3} = \emptyset
$$
on $U_{2}^{(3)}$.
Thus, we get $K_{W_{3}} \cdot E'_{2} \cdot E_{3}=0$ on $W_{3}$.
\end{proof}

Lemmas~\ref{theorem:zero-intersection} and \ref{theorem:restriction-is-contracted} imply that divisor $E'_{2}$ is
$\varphi_{\mathcal{L}_{W_{3}}}$-exceptional and $o:=\varphi_{\mathcal{L}_{W_{3}}}(E'_{2})$ is a point on $X$. Moreover,
since the discrepancy $a(E'_{2}, X)$ is zero, $o \in X$ is a non-$\mathrm{cDV}$ singular point by Theorem~\ref{theorem:Reid-CDV}.
Proposition~\ref{theorem:singularities-of-T-variety} is completely proved.
\end{proof}

Now we turn to the second case in Lemma~\ref{theorem:special-scrolls}.

\begin{proposition}
\label{theorem:C-variety-from-scroll} In the assumptions of
Lemma~\ref{theorem:special-scrolls}, if $W =
\mathrm{Proj}\left(\mathcal{O}_{\mathbb{P}^{1}}(6)\oplus\mathcal{O}_{\mathbb{P}^{1}}(2)\oplus\mathcal{O}_{\mathbb{P}^{1}}\right)$,
then $-K_{X}^{3} = 70$ and $X = X_{70}$ is that constructed in
Proposition~\ref{theorem:contraction-to-curve}.
\end{proposition}

\begin{proof}
For $W =
\mathrm{Proj}\left(\mathcal{O}_{\mathbb{P}^{1}}(6)\oplus\mathcal{O}_{\mathbb{P}^{1}}(2)\oplus\mathcal{O}_{\mathbb{P}^{1}}\right)$
we have $\varphi_{|-K_{W}|}(W) = \mathbb{P}(6,4,1,1)$ (see
\cite[Ch. 4, Remark 4.2]{Iskovskikh-anti-canonical-models}). Then
from Remark~\ref{remark:linear-systems} and the assumption
$\dim\mathcal{L}_{W} \in\{35, 36, 37\}$ we deduce that the initial
Fano threefold $X$ must be an image of $\mathbb{P}(6,4,1,1)
\subset \mathbb{P}^{38}$ under the birational projection from a
point, a line or a plane. Let $\pi: \mathbb{P}(6,4,1,1)
\dashrightarrow X$ be this projection.

\begin{lemma}
\label{theorem:projection-from-point}
If $\pi$ is the projection from a point, then $X = X_{70}$ is that constructed in Proposition~\ref{theorem:contraction-to-curve}.
\end{lemma}

\begin{proof}
Let $O$ be the center of the projection $\pi$. Then $O$ belongs to
$\mathbb{P}= \mathbb{P}(6,4,1,1)$ because otherwise $\pi$ is an
isomorphism which is impossible because $\mathbb{P} \subset
\mathbb{P}^{38}$ is not degenerate. Moreover, if $O$ is a smooth
point on $\mathbb{P}$, then $-K_{X}^{3}=71$ which is also
impossible.

In the notation of Example~\ref{example:examp-2}, if $O$ is
distinct from $P$ and $Q$, then $O\in\mathbb{P}$ is a
$\mathrm{cA_{1}}$ point. By Lemma~\ref{theorem:non-simple-case} we
get $X = X_{70}$.

Now, if $O$ is either $P$ or $Q$, then general hyperplane section
$S$ of $\mathbb{P}$ through $O$ is a $\mathrm{K3}$ surface with
singularities worse than Du Val. Indeed, otherwise either $P$ or
$Q$ is a $\mathrm{cDV}$ point which is impossible (see
Proposition~\ref{theorem:inductive-CDV},
Example~\ref{example:examp-2} and Theorem~\ref{theorem:Reid-CDV}).
In particular, we have $\kappa(S)<0$ for the Kodaira dimension of
the surface $S$.

On the other hand, the projection $\pi$ induces a birational map $S  \dashrightarrow S'$ on general surface
$S'\in |-K_{X}|$. This implies that $\kappa(S) = \kappa(S')$. But $S'$ is a $\mathrm{K3}$
surface with canonical singularities by
Theorem~\ref{theorem:elefant}. Thus, we get $0 = \kappa(S') = \kappa(S) < 0$, a contradiction.
\end{proof}

By Lemma~\ref{theorem:projection-from-point} it remains to consider the cases when $\pi$ is either the projection from some line $\gamma$ or some
plane $\Omega$. We will show that in each of these cases $-K_{X}^{3} \notin \{66,68,70\}$.

Let us consider the projection from line $\gamma$ first. Since
$\mathbb{P}=\mathbb{P}(6,4,1,1)$ is an intersection of quadrics
(see Proposition~\ref{theorem:free-antican-system-2}), its
intersection with $\gamma$ is either a set of $\leqslant 2$ points
or the whole $\gamma$. Note that $\gamma \cap \mathbb{P} \ne
\emptyset$ because $\mathbb{P} \subset \mathbb{P}^{38}$ is not
degenerate.

\begin{lemma}
\label{theorem:projection-from-line-1}
The line $\gamma$ is not contained in $\mathbb{P}$.
\end{lemma}

\begin{proof}
Suppose that $\gamma \subset \mathbb{P}$. Then we have $-K_{\mathbb{P}} \cdot \gamma = 1$. On the other hand, on
$\mathbb{P}$ we have $-K_{\mathbb{P}} \sim \mathcal{O}_{\mathbb{P}}(12)$ (see \cite{Iano-Fletcher}). Hence
$$
\mathcal{O}_{\mathbb{P}}(1)\cdot \gamma = \frac{1}{12}
$$
which implies that $\gamma$ passes through a singular point on
$\mathbb{P}$.

If $\gamma$ contains a $\mathrm{cDV}$ point, then by
Lemma~\ref{theorem:unique-line} $\gamma$ coincides with the
singular locus of $\mathbb{P}$. As in the proof of
Lemma~\ref{theorem:projection-from-point}, for general hyperplane
section $S$ of $\mathbb{P}$ through $\gamma$ we have $\kappa(S) <
0$ and for its birational image $S'$ on $X$ we have $\kappa(S') =
0$, a contradiction. In the same way we obtain a contradiction
when
$\gamma$ contains one of the non-$\mathrm{cDV}$ points of
$\mathbb{P}$.
\end{proof}

Thus, $\gamma$ intersects $\mathbb{P}$ at $\leqslant 2$ points.

\begin{lemma}
\label{theorem:projection-from-line-2}
In the above assumptions, $\gamma$ contains only smooth points of $\mathbb{P}$.
\end{lemma}

\begin{proof}
In the notation of Example~\ref{example:examp-2}, if $\gamma$
contains a singular point on $\mathbb{P}$ distinct from $P$ and
$Q$, then by Lemma~\ref{theorem:non-simple-case} $X$ is an image
under the birational projection of the threefold $X_{70}$ from
some point $O$ on $X_{70}$.

If $O$ is a smooth point, then $-K_{X}^{3} = 69$ which is
impossible.

Thus, by Proposition~\ref{theorem:singularities-of-X} $O \in X_{70}$ is the unique non-$\mathrm{cDV}$ point.
Now, as in the proof of Lemma~\ref{theorem:projection-from-point},
for general hyperplane section $S$ of $X_{70}$ through $O$ we have $\kappa(S) < 0$ and for its birational image $S'$ on $X$ we have $\kappa(S') = 0$.

The obtained contradiction implies that $\gamma$ contains only $P$ and $Q$. Again, as in the proof of Lemma~\ref{theorem:projection-from-point},
for general hyperplane section $S$ of $\mathbb{P}$ through $\gamma$ we have $\kappa(S) < 0$ and for its birational image $S'$ on $X$ we have $\kappa(S') = 0$, a contradiction.
\end{proof}

Thus, $\gamma$ intersects $\mathbb{P}$ at $\leqslant 2$ smooth points.
\begin{lemma}
\label{theorem:projection-from-line-3}
In the above assumptions, we have $-K_{X}^{3} = 72$.
\end{lemma}

\begin{proof}
Let $S$ be a general hyperplane section of $\mathbb{P}$ through $\gamma$. Then $S$ is a non-singular $\mathrm{K3}$ surface
because $\gamma \cap \mathbb{P}$ consists of $\leqslant 2$ smooth points.
Furthermore, $\pi$ induces a birational map $\chi: S \dashrightarrow S'$ on general surface $S'\in |-K_{X}|$ having
canonical singularities (see Theorem~\ref{theorem:elefant}). This implies that $S$ is the minimal resolution of $S'$ and
$\chi$ is a regular morphism.

Projection $\pi$ is given by the linear system $\mathcal{L}$ of
all hyperplane sections of $\mathbb{P}$ through $\gamma$. Set
$\mathcal{L}_{S}: = \mathcal{L}|_{S}$ and $\mathcal{L}_{S'}: =
|-K_{X}||_{S'}$. Then we have $\mathcal{L}_{S} =
\chi^{*}\mathcal{L}_{S'}$ and in particular $(L)^{2} = (L')^{2}$
for $L\in\mathcal{L}_{S}$ and $L'\in\mathcal{L}_{S'}$.

On the other hand, we have $(L')^{2} = -K_{X}^{3}$ and $(L)^{2} = -K_{\mathbb{P}}^{3}$. Thus, we obtain
$$
-K_{X}^{3}=(L')^{2}=(L)^{2}=72.
$$
\end{proof}

From Lemma~\ref{theorem:projection-from-line-3} we get a
contradiction with $-K_{X}^{3}\in\{66,68,70\}$. Now we turn to the
final case when $\pi$ is the projection from plane $\Omega$.

\begin{lemma}
\label{theorem:projection-from-plane-1} The plane $\Omega$ is not
contained in $\mathbb{P}(6,4,1,1)$.
\end{lemma}

\begin{proof}
Suppose that $\Pi \subset \mathbb{P}(6,4,1,1)$. Then, as in the
proof of Lemma~\ref{theorem:projection-from-line-1}, any line on
$\Pi$ passes through a singular point on $\mathbb{P}(6,4,1,1)$.
Now repeating word by word the arguments from the proof of
Lemma~\ref{theorem:projection-from-line-1} we obtain a
contradiction.
\end{proof}

Since $\mathbb{P} = \mathbb{P}(6,4,1,1)$ is an intersection of
quadrics, by Lemma~\ref{theorem:projection-from-plane-1} the
intersection of $\mathbb{P}$ with $\Omega$ is either a set of
$\leqslant 4$ points or a (non-reduced, reduced or irreducible)
conic. Again we note that $\Omega \cap \mathbb{P} \ne \emptyset$
because $\mathbb{P} \subset \mathbb{P}^{38}$ is not degenerate.

\begin{lemma}
\label{theorem:projection-from-plane-01} If $\Omega \cap
\mathbb{P}$ is a finite set, then it contains only smooth points
of $\mathbb{P}$.
\end{lemma}

\begin{proof}
In the notation of Example~\ref{example:examp-2}, if $\Omega$
contains a singular point on $\mathbb{P}$ distinct from $P$ and
$Q$, then by Lemma~\ref{theorem:non-simple-case} $X$ is an image
under the birational projection of the threefold $X_{70}$ from
some line $\gamma'$ which is not contained in $X_{70}$ and
$\gamma' \cap X_{70} \ne \emptyset$.

If $\gamma'$ does not pass through the singular point on $X_{70}$,
then repeating word by word the arguments from the proof of
Lemma~\ref{theorem:projection-from-line-3} we obtain that
$-K_{X}^{3} = -K_{X_{70}}^{3}$. On the other hand, we must have
$-K_{X}^{3} < -K_{X_{70}}^{3}$, since the degree must decrease, a
contradiction.

Thus, $\gamma'$ contains unique singular point $O$ on $X_{70}$
(see Proposition~\ref{theorem:singularities-of-X}). As in the
proof of Lemma~\ref{theorem:projection-from-point}, for general
hyperplane section $S$ of $X_{70}$ through $O$ we have $\kappa(S)
< 0$ and for its birational image $S'$ on $X$ we have $\kappa(S')
= 0$.

The obtained contradiction implies that $\Omega$ contains only $P$
and $Q$. Again, as in the proof of
Lemma~\ref{theorem:projection-from-point}, for general hyperplane
section $S$ of $\mathbb{P}$ through $\Omega$ we have $\kappa(S) <
0$ and for its birational image $S'$ on $X$ we have $\kappa(S') =
0$, a contradiction.
\end{proof}

Thus, if $\Omega$ intersects $\mathbb{P}$ by $\leqslant 4$ points,
then by Lemma~\ref{theorem:projection-from-plane-01} and the
arguments from the proof of
Lemma~\ref{theorem:projection-from-line-3} we obtain that
$-K_{X}^{3} = 72$ which is a contradiction with
$-K_{X}^{3}\in\{66,68,70\}$.

Suppose that $\Omega$ intersects $\mathbb{P}$ by some conic $C$.

\begin{lemma}
\label{theorem:projection-from-plane-2}
In the above notation, the set $C \cap \mathrm{Sing}(\mathbb{P})$ is non-empty and consists of $\mathrm{cA_{1}}$ points.
\end{lemma}

\begin{proof}
Since $-K_{\mathbb{P}}\cdot C = 2$ and $-K_{\mathbb{P}}\sim\mathcal{O}_{\mathbb{P}}(12)$, on $\mathbb{P}$ we have
$$
\mathcal{O}_{\mathbb{P}}(1)\cdot C = \frac{1}{6}.
$$
This implies that $C$ passes through a singular point on
$\mathbb{P}$.

Furthermore, in the notation of Example~\ref{example:examp-2}, if $C$ contains either $P$ or $Q$, then as in the proof of
Lemma~\ref{theorem:projection-from-point}, for general hyperplane section $S$ of $\mathbb{P}$ through $\Omega$ we have
$\kappa(S) < 0$ and for its birational image $S'$ on $X$ we have $\kappa(S') = 0$, a contradiction.
\end{proof}

From Lemmas~\ref{theorem:projection-from-plane-2} and
\ref{theorem:non-simple-case} we deduce that $X$ is an image under
the birational projection of the threefold $X_{70}$ from some line
$\gamma'$ on $X_{70}$.

\begin{lemma}
\label{theorem:projection-from-plane-3} In the above notation, the
line $\gamma'$ passes through the singular point on $X_{70}$.
\end{lemma}

\begin{proof}
If $\gamma'$ does not pass through the singular point on $X_{70}$,
then by construction $X_{70}$ is isomorphic to $\mathbb{P}$ near
$\gamma'$. This implies that there is a line on $\mathbb{P}$ not
passing through the singular points which is impossible (see the
proof of Lemma~\ref{theorem:projection-from-line-1}).
\end{proof}

Thus, $\gamma'$ contains unique singular point $O$ on $X_{70}$
(see Proposition~\ref{theorem:singularities-of-X}). Then, as in
the proof of Lemma~\ref{theorem:projection-from-point}, for
general hyperplane section $S$ of $X_{70}$ through $O$ we have
$\kappa(S) < 0$ and for its birational image $S'$ on $X$ we have
$\kappa(S') = 0$, a contradiction.

Proposition~\ref{theorem:C-variety-from-scroll} is completely
proved.
\end{proof}

Now, for $W_{\eta} \simeq \mathbb{P}^{2}$ and $H|_{W_{\eta}}
\simeq \mathcal{O}_{\mathbb{P}^{2}}(2)$ according to Proposition
8.7 in \cite{Prokhorov-degree} \emph{there is a
$(K_{W}+\mathcal{L}_{W})$-crepant birational map $W
\dashrightarrow W_{0}$ onto a $\mathbb{P}^{2}$-bundle $W_{0}$ over
$V$}. Then the pair $(W_{0}, \mathcal{L}_{0})$ is canonical and
the linear system $\mathcal{L}_{0} \subset |-K_{W_{0}}|$ gives a
birational map $\varphi_{\mathcal{L}_{0}}: W_{0} \dashrightarrow
X$ (see Remarks~\ref{remark:singularities-of-pairs} and
\ref{remark:linear-systems}). Again by
Lemma~\ref{theorem:special-scrolls} either $W_{0} =
\mathrm{Proj}\left(\mathcal{O}_{\mathbb{P}^{1}}(5)\oplus\mathcal{O}_{\mathbb{P}^{1}}(2)\oplus\mathcal{O}_{\mathbb{P}^{1}}\right)$
or
$\mathrm{Proj}\left(\mathcal{O}_{\mathbb{P}^{1}}(6)\oplus\mathcal{O}_{\mathbb{P}^{1}}(2)\oplus\mathcal{O}_{\mathbb{P}^{1}}\right)$
and we repeat the arguments from the proofs of
Propositions~\ref{theorem:T-variety} and
\ref{theorem:C-variety-from-scroll} to get that the initial Fano
threefold $X$ is either $X_{66}$ or $X_{70}$ (see
Remarks~\ref{remark:rank-of-Pic} and \ref{remark:T-variety}).

Finally, if $W_{\eta} \simeq \mathbb{P}^{1} \times \mathbb{P}^{1}$, then Proposition 9.4 in \cite{Prokhorov-degree} can
be applied to the initial Fano threefold $X$ having $-K_{X}^{3}\in \{66,68,70\}$.
Again \emph{there is a $(K_{W}+\mathcal{L}_{W})$-crepant birational map $W \dashrightarrow W_{0}$ onto a
$\mathbb{P}^{2}$-bundle $W_{0}$ over $V$} and the previous arguments work.

Theorem~\ref{theorem:main} is completely proved in the case when $V$ is a curve.

\section{Proof of Theorem~\ref{theorem:main} in the case when $V$ is a surface}
\label{section:surface}

We use notation and assumptions from Section~\ref{section:reduction}. In the present section $g: W \longrightarrow V$ is a
$(K_{W}+\mathcal{H}_{W})$-negative extremal contraction on the surface $V$. Following the lines in the proof of Proposition 10.3
in \cite{Prokhorov-degree} we will show that this situation is impossible.

\begin{lemma}[{\cite[Lemma 10.1]{Prokhorov-degree}}]
\label{theorem:P-1-fibration}
The surface $V$ is non-singular and the contraction $g: W \longrightarrow V$ is a $\mathbb{P}^{1}$-bundle.
\end{lemma}

For the fibre $F$ of the contraction $g$ we have $-K_{W} \cdot F = 2$. This and the equivalence
$$
K_{W} + \mathcal{H}_{W} + D_{W} \sim 0
$$
imply that divisors $H \in \mathcal{H}_{W}$ and $D_{W}$ are the sections of $g$.

Set $\mathcal{E}:=g_{*}\mathcal{O}_{W}(H)$. Then $\mathcal{E}$ is a rank $2$ vector bundle and
$W \simeq \mathbb{P}(\mathcal{E})$. By Lemma 4.4 in \cite{Prokhorov-degree} we have

\begin{equation}
\label{equation:dimension-of-H}
\begin{array}{c}
\dim |H| + 1 = h^{0}(V, \mathcal{E}) = \chi(V, \mathcal{E}).
\end{array}
\end{equation}

\begin{proposition}
\label{theorem:good-estimate}
In the above notation, we have $\dim |H| \in \{32,33,34,35\}$.
\end{proposition}

\begin{proof}
According to Remark~\ref{remark:inequalities} we have $\dim |H| \geqslant 32$.

Suppose that $\dim |H| \geqslant 36$. Then
it follows from the proof of Proposition 10.3 in \cite{Prokhorov-degree} that
$W = \mathbb{P}(\mathcal{O}_{\mathbb{P}^{2}}(3)\oplus\mathcal{O}_{\mathbb{P}^{2}}(6))$. In this case the anti-canonical
image of $W$ is $\mathbb{P}(3,1,1,1)$ (see Example~\ref{example:examp-1}).

From Remark~\ref{remark:linear-systems} and the assumption
$\dim\mathcal{L}_{W} \in\{35, 36, 37\}$ we deduce that the initial
Fano threefold $X$ must be an image of $\varphi_{|-K_{W}|}(W) =
\mathbb{P}(3,1,1,1) \subset \mathbb{P}^{38}$ under the birational
projection from a point, a line or a plane. Let $\pi:
\mathbb{P}(3,1,1,1) \dashrightarrow X$ be this projection.

\begin{lemma}
\label{theorem:projection-1} If $\pi$ is the projection from a
point, then $K_{X}^{3} \notin \{66,68,70\}$.
\end{lemma}

\begin{proof}
Let $O$ be the center of the projection $\pi$. Then $O$ belongs to
$\mathbb{P}=\mathbb{P}(3,1,1,1)$ because otherwise $\pi$ is an
isomorphism which is impossible because $\mathbb{P} \subset
\mathbb{P}^{38}$ is not degenerate. If $O$ is a smooth point on
$\mathbb{P}$, then $-K_{X}^{3}=71$.

Now let $O$ be the unique singular point on $\mathbb{P}$. Then
singularity $O \in \mathbb{P}$ is non-$\mathrm{cDV}$ (see
Example~\ref{example:examp-1} and Theorem~\ref{theorem:Reid-CDV}).
This implies that general hyperplane section $S$ of $\mathbb{P}$
through $O$ is a $\mathrm{K3}$ surface with singularities worse
than Du Val (see Proposition~\ref{theorem:inductive-CDV}). In
particular, we have $\kappa(S) < 0$ for the Kodaira dimension of
$S$.

On the other hand, the projection $\pi$ induces a birational map $S \dashrightarrow S'$ on general surface $S' \in |-K_{X}|$. This implies that $\kappa(S) = \kappa(S')$. But the surface $S'$ has canonical singularities by Theorem~\ref{theorem:elefant}. Thus, we get $0 = \kappa(S') = \kappa(S) < 0$, a contradiction.
\end{proof}

\begin{lemma}
\label{theorem:projection-2}
If $\pi$ is the projection from a line, then $-K_{X}^{3} \notin \{66,68,70\}$.
\end{lemma}

\begin{proof}
Let $\gamma$ be the center of the projection $\pi$. Then $\gamma$
must intersect $\mathbb{P} = \mathbb{P}(3,1,1,1)$ because
otherwise $\pi$ is an isomorphism which is impossible because
$\mathbb{P} \subset \mathbb{P}^{38}$ is not degenerate. Moreover,
$\gamma$ is not contained in $\mathbb{P}$, since divisor
$-K_{\mathbb{P}}$ is divisible in $\mathrm{Pic}(\mathbb{P})$, and
intersects $\mathbb{P}$ at $\leqslant 2$ points because
$\mathbb{P}$ is an intersection of quadrics (see
Proposition~\ref{theorem:free-antican-system-2}).

If $\gamma$ contains unique singular point on $\mathbb{P}$, then
as in the proof of Lemma~\ref{theorem:projection-1} for general
hyperplane section $S$ of $\mathbb{P}$ through $\gamma$ we have
$\kappa(S) < 0$ and for its birational image $S'$ on $X$ we have
$\kappa(S') = 0$. Thus, we get $0 = \kappa(S') = \kappa(S) < 0$, a
contradiction.

Finally, if $\gamma$ contains only smooth points on $\mathbb{P}$, then general hyperplane section $S$ of $\mathbb{P}$
through $\gamma$ is a non-singular $\mathrm{K3}$ surface.
Furthermore, the projection $\pi$ induces a birational map $\eta: S \dashrightarrow S'$ on general surface $S' \in |-K_{X}|$. The surface $S'$ has canonical singularities by Theorem~\ref{theorem:elefant}. This implies that $S$ is the minimal resolution of $S'$ and $\eta$ is a regular morphism.

Projection $\pi$ is given by the linear system $\mathcal{L}$ of
all hyperplane sections of $\mathbb{P}$ through $\gamma$. Set
$\mathcal{L}_{S}: = \mathcal{L}|_{S}$ and $\mathcal{L}_{S'}: =
|-K_{X}||_{S'}$. Then we have $\mathcal{L}_{S} =
\eta^{*}\mathcal{L}_{S'}$ and in particular $(L)^{2} = (L')^{2}$
for $L\in\mathcal{L}_{S}$ and $L'\in\mathcal{L}_{S'}$.

On the other hand, we have $(L')^{2} = -K_{X}^{3}$ and $(L)^{2} = -K_{\mathbb{P}}^{3}$. Thus, we get
$$
-K_{X}^{3}=(L')^{2}=(L)^{2}=72.
$$
\end{proof}

\begin{lemma}
\label{theorem:projection-3}
If $\pi$ is the projection from a plane, then $-K_{X}^{3} \notin \{66,68,70\}$.
\end{lemma}

\begin{proof}
Let $\Omega$ be the center of the projection $\pi$. Then $\Omega$
must intersect $\mathbb{P} = \mathbb{P}(3,1,1,1)$ because
otherwise $\pi$ is an isomorphism which is impossible because
$\mathbb{P} \subset \mathbb{P}^{38}$ is not degenerate. Moreover,
$\Omega$ is not contained in $\mathbb{P}$, since divisor
$-K_{\mathbb{P}}$ is divisible in $\mathrm{Pic}(\mathbb{P})$. Then
by Proposition~\ref{theorem:free-antican-system-2} $\Omega$
intersects $\mathbb{P}$ either at $\leqslant 4$ points or a conic.

If $\Omega$ intersects $\mathbb{P}$ at $\leqslant 4$ points, then
we repeat word by word the arguments from the proof of
Lemma~\ref{theorem:projection-2} to obtain that $-K_{X}^{3} = 72$.

Now let $\Omega \cap \mathbb{P}$ be a conic $C$. Note that $C$ is
reduced and irreducible because divisor $-K_{\mathbb{P}}$ is
divisible in $\mathrm{Pic}(\mathbb{P})$. The image of $\mathbb{P}$
under the isomorphism $\varphi_{|-\frac{1}{2}K_{\mathbb{P}}|}$ is
the cone over del Pezzo surface of degree $9$. This implies that
$\varphi_{|-\frac{1}{2}K_{\mathbb{P}}|}(C)$ is the generatrix of
this cone.

Hence conic $C$ contains unique singular point on $\mathbb{P}$.
Then, as in the proof of Lemma~\ref{theorem:projection-1} for
general hyperplane section $S$ of $\mathbb{P}$ through $\Omega$ we
have $\kappa(S) < 0$ and for its birational image $S'$ on $X$ we
have $\kappa(S') = 0$. Thus, we get $0 = \kappa(S') = \kappa(S) <
0$, a contradiction.
\end{proof}

From Lemmas~\ref{theorem:projection-1}--\ref{theorem:projection-3}
we derive a contradiction with $\dim |H| \geqslant 36$.
Proposition~\ref{theorem:good-estimate} is completely proved.
\end{proof}

\begin{lemma}[\cite{Prokhorov-degree}]
\label{theorem:rational-base}
In the above notation, the surface $V$ is rational.
\end{lemma}

\begin{proof}
From the Leray spectral sequence and Kawamata--Vieweg Vanishing Theorem
we deduce that $H^{1}(V, \mathcal{O}_{V}) = H^{1}(W, \mathcal{O}_{W}) = 0$.

Furthermore, the surface $V$ is dominated by an irreducible component of the surface $D_{W}$ of negative Kodaira dimension (see Remark~\ref{remark:birationally-rulled-components}).
Since $V$ is non-singular by Lemma~\ref{theorem:P-1-fibration}, this implies that the surface $V$ is rational.
\end{proof}

\begin{remark}
\label{remark:minimal-rational-base} By \cite[Lemma
10.4]{Prokhorov-degree} and Lemma~\ref{theorem:rational-base} we
may assume that either $V \simeq \mathbb{P}^{2}$ or
$\mathbb{F}_{n}$ for some $n \ne 1$. Moreover, by \cite[Lemma
10.12]{Prokhorov-degree} in the last case we have $n \leqslant 4$.
\end{remark}

Let us denote by $c_{i}:=c_{i}(\mathcal{E})$ the $i$-th Chern class of the rank $2$ vector bundle $\mathcal{E}$, $i=1$, $2$. Then by the relative Euler
exact sequence we have
$$
-K_{W} \sim 2H + g^{*}(-K_{V} - c_{1})
$$
and by the Hirsh formula
$$
H^{2} \equiv H \cdot g^{*}c_{1} - c_{2}, \qquad H^{3} = c_{1}^{2} - c_{2}.
$$
In particular, we get the following formula:
\begin{equation}
\label{equation:formula-for-degree}
\begin{array}{c}
-K_{W}^{3} = 6K_{V}^{2} + 2c_{1}^{2} - 8c_{2}.
\end{array}
\end{equation}

Recall also the Riemann--Roch formula for rank $2$ vector bundles over a rational surface:
\begin{equation}
\label{equation:Rieman-Roch-formula}
\begin{array}{c}
\chi(\mathcal{E}) = \frac{1}{2}\left(c_{1}^{2}-2c_{2}-K_{V} \cdot c_{1}\right)+2.
\end{array}
\end{equation}

We will need the following lemmas:

\begin{lemma}[\cite{Prokhorov-degree}]
\label{theorem:estimate-for-1-class}
In the above notation and assumptions, the inequality
$$
c_{1} \cdot B \leqslant -3K_{V} \cdot B
$$
holds for every nef divisor $B$ on $V$.
\end{lemma}

\begin{proof}
Since
$$
D_{W} \sim H + g^{*}(-K_{V} - c_{1}),
$$
for the nef divisor $B$ on the surface $V$ we have
$$
0 \leqslant -K_{W} \cdot D_{W} \cdot g^{*}B = 2H^{2} \cdot g^{*}B +3(-K_{V}-c_{1}) \cdot B =
2c_{1} \cdot B + 3(-K_{V}-c_{1}) \cdot B = -3K_{V} \cdot B - c_{1} \cdot B.
$$
\end{proof}

\begin{lemma}[\cite{Prokhorov-degree}]
\label{theorem:when-anti-can-class-nef}
In the above notation and assumptions, let $B > 0$ be a nef divisor on the surface $V$ such that $-K_{W} \cdot D_{W} \cdot g^{*}B = 0$. Then
divisor $-K_{W}$ is nef.
\end{lemma}

\begin{proof}
Let $R$ be a horizontal curve on $W$ such that $-K_{W} \cdot R < 0$. Then
$$
D_{W} \cdot R = (-K_{W} - H) \cdot R < 0.
$$
This implies that $R \subseteq L \cap D_{W}$ for a general surface $L \in |-K_{W}|$. But from the condition of lemma we
deduce that the intersection $L \cap D_{W}$ is composed of fibres of the morphism $g$, a contradiction.
\end{proof}

\begin{lemma}
\label{theorem:property-of-1-class}
In the above notation and assumptions, the class $- K_{V}+c_{1}$ is nef on $V$.
\end{lemma}

\begin{proof}
The linear system $|-K_{W}|$ does not have fixed components and divisor $H$ is nef (see
Remarks~\ref{remark:singularities-of-pairs} and \ref{remark:linear-systems}).
Then for every irreducible curve $C$ on the
surface $V$ we have
$$
0 \leqslant -K_{W} \cdot H \cdot g^{*}C = 2H^{2} \cdot g^{*}C + (-K_{V}-c_{1}) \cdot C =
2c_{1} \cdot C + (-K_{V}-c_{1}) \cdot C = (- K_{V}+c_{1}) \cdot C.
$$
\end{proof}

\begin{lemma}[{\cite[Lemma 10.6]{Prokhorov-degree}}]
\label{theorem:restriction-on-movable-curve}
In the above notation and assumptions, let $Z \subset V$ be an irreducible rational curve such that $\dim |Z| > 0$ and let
$\mathcal{E}|_{Z} \simeq \mathcal{O}_{\mathbb{P}^{1}}(d_{1}) \oplus \mathcal{O}_{\mathbb{P}^{1}}(d_{2})$. Then
$|d_{1} - d_{2}| \leqslant 2 + Z^{2}$. Moreover, if $V \simeq \mathbb{F}_{n}$, $n \geqslant 0$, and $Z$ is the tautological
section on $V$, then $|d_{1} - d_{2}| =  c_{1}(\mathcal{E}) \cdot Z$.
\end{lemma}

\begin{proof}
Let $m: = |d_{1} - d_{2}|$ and $G: = g^{-1}(Z)$. Then $G \simeq \mathbb{F}_{m}$ and for the negative section $\Sigma$ on $G$ we have
\begin{equation}
\label{equation:intersections}
\begin{array}{c}
-2 + m = -2 -\Sigma^{2} = K_{G} \cdot \Sigma = K_{W} \cdot \Sigma + G \cdot \Sigma = K_{W} \cdot \Sigma + Z^{2}.
\end{array}
\end{equation}
Since the linear system $|-K_{W}|$ does not have fixed components, we get
$K_{W} \cdot \Sigma \leqslant 0$ and hence $m \leqslant 2 + Z^{2}$.

Moreover, if $Z$ is the tautological section on
$V \simeq \mathbb{F}_{n}$, $n \geqslant 0$, then $Z^{2} = n$ and
$$
K_{W} \cdot \Sigma = g^{*}(K_{V} + c_{1}(\mathcal{E})) \cdot \Sigma = (K_{V} + c_{1}(\mathcal{E})) \cdot Z = -n - 2 + c_{1}(\mathcal{E}) \cdot Z.
$$
From \eqref{equation:intersections} we obtain $m = c_{1}(\mathcal{E}) \cdot Z$.
\end{proof}

Recall that by Remark~\ref{remark:minimal-rational-base} we have either $V = \mathbb{P}^{2}$ or $\mathbb{F}_{n}$ where $n \in \{0,2,3,4\}$. We will
show that each of these cases is impossible.

\begin{proposition}
\label{theorem:P-2-case} In the above notation and assumptions, we
have $V \ne \mathbb{P}^{2}$.
\end{proposition}

\begin{proof}
Suppose that $V = \mathbb{P}^{2}$.

We have $H^{2i}(\mathbb{P}^{2}, \mathbb{Z}) \simeq \mathbb{Z}$, $i=1$, $2$. Hence we may assume that $c_{1}$ and $c_{2}$ are integers.
Then by Lemma~\ref{theorem:estimate-for-1-class} $0 \leqslant c_{1} \leqslant 9$.

\begin{lemma}
\label{theorem:non-decomposable}
In the above notation, the rank $2$ vector bundle $\mathcal{E}$ is not decomposable.
\end{lemma}

\begin{proof}
Suppose that $\mathcal{E}$ is decomposable. Then
$\mathcal{E} \simeq \mathcal{O}_{\mathbb{P}^{2}}(a) \oplus \mathcal{O}_{\mathbb{P}^{2}}(a+b)$ for some $b \geqslant 0$. Since
$H$ is nef we also have $a \geqslant 0$.

Furthermore, we have $c_{1} = 2a + b$ and $c_{2} = a^{2}+ab$. Thus, from \eqref{equation:Rieman-Roch-formula} we get
$$
\chi(V,\mathcal{E}) = \frac{1}{2}\left(2a^{2}+2ab+b^{2}+6a+3b\right)+2.
$$

By \eqref{equation:dimension-of-H} and Lemma~\ref{theorem:good-estimate} $\chi(V,\mathcal{E}) \in \{33,34,35,36\}$.
Then from the previous expression and the estimates $a$, $b \geqslant 0$,
$0 \leqslant c_{1} \leqslant 9$ it is easy to see that the only possibility is $\chi(V,\mathcal{E})=36$ and $(a, b) = (4, 1)$.

In particular, we have $c_{1} = -3K_{V}$. Then it follows from the proof of
Lemma~\ref{theorem:estimate-for-1-class}
that
$$
-K_{W} \cdot D_{W} \cdot g^{*}B = 0
$$
for every curve $B$ on the surface $V$. From Lemma~\ref{theorem:when-anti-can-class-nef} we deduce that $-K_{W}$ is nef and
hence $W$ is a weak Fano threefold.

On the other hand, for $(a, b) = (4, 1)$ we have $c_{2} = 20$. Then by \eqref{equation:formula-for-degree} $-K_{W}^{3} = 56$ and
by the Riemann--Roch formula and Kawamata--Vieweg Vanishing Theorem we get
$$
\dim |-K_{W}| = -\frac{1}{2}K_{W}^{3} + 2 = 30.
$$
But according to Remark~\ref{remark:inequalities} $\dim |-K_{W}| \geqslant 35$, a contradiction.
\end{proof}

\begin{remark}
\label{remark:when-E-is-decomposable}
It follows from the arguments in the proof of Proposition 10.3 in \cite{Prokhorov-degree} that if $c_{1} = 9$, then
$\mathcal{E}$ is decomposable. Thus, by Lemma~\ref{theorem:non-decomposable} we have $0 \leqslant c_{1} \leqslant 8$.
\end{remark}

\begin{lemma}
\label{theorem:even}
The class $c_{1}$ is even.
\end{lemma}

\begin{proof}
Suppose that $c_{1}$ is odd. Then by
Remark~\ref{remark:when-E-is-decomposable} we have $c_{1} = 2m-3$
for some $2 \leqslant m \leqslant 5$. Since by
\eqref{equation:dimension-of-H} and
Lemma~\ref{theorem:good-estimate} we have $\chi(V,\mathcal{E}) \in
\{33,34,35,36\}$, from \eqref{equation:Rieman-Roch-formula} we get
$2m^{2} - 3m - c_{2} \geqslant 31$. Then
$$
c_{1}(\mathcal{E}(-m))=-3, \qquad c_{2}(\mathcal{E}(-m)) = c_{2} - m^{2} + 3m \leqslant m^{2} - 31 < 0.
$$

By the Riemann--Roch formula and Serre Duality we obtain
$$
h^{0}(\mathcal{E}(-m))+h^{0}(\mathcal{E}(-m)\otimes \det \mathcal{E}(-m)^{*}\otimes \mathcal{O}(-3)) \geqslant
\chi(V, \mathcal{E}(-m)) \geqslant 1.
$$
Thus, $H^{0}(\mathcal{E}(-m)) \ne 0$. Let $s \in H^{0}(\mathcal{E}(-m))$ be a nonzero section. If $s$ does not
vanish anywhere, then there is an embedding $\mathcal{O} \hookrightarrow \mathcal{E}(-m)$ and $\mathcal{E}(-m)$ is
decomposable. This contradicts Lemma~\ref{theorem:non-decomposable}.

Let $\emptyset \ne Y \subset \mathbb{P}^{2}$ be the zero locus of $s$. Since $c_{2}(\mathcal{E}(-m)) < 0$, we have $\dim Y = 1$.
Take a general line $Z \subset \mathbb{P}^{2}$ and let the intersection $Z \cap Y$ consists of $k$ points. Since
$c_{1}(\mathcal{E}(-m))=-3$, we have
$$
\mathcal{E}(-m)|_{Z} = \mathcal{O}_{\mathbb{P}^{1}}(k) \oplus \mathcal{O}_{\mathbb{P}^{1}}(-k-3).
$$
Then by Lemma~\ref{theorem:restriction-on-movable-curve} we get $2k + 3 \leqslant 3$.
Hence $k=0$ and $Y = \emptyset$, a contradiction.
\end{proof}

From Lemma~\ref{theorem:even} and
Remark~\ref{remark:when-E-is-decomposable} we deduce that $c_{1} =
2m-2$ for some $1\leqslant m \leqslant 5$. Since by
\eqref{equation:dimension-of-H} and
Lemma~\ref{theorem:good-estimate} we have $\chi(V,\mathcal{E}) \in
\{33,34,35,36\}$, from \eqref{equation:Rieman-Roch-formula} we get
$2m^{2}-m-c_{2} \geqslant 32$. Then
$$
c_{1}(\mathcal{E}(-m))=-2, \qquad c_{2}(\mathcal{E}(-m)) = c_{2} - m^{2} + 2m \leqslant m^{2} + m  - 32 < 0.
$$

As in the proof of Lemma~\ref{theorem:even}, we obtain  $H^{0}(\mathcal{E}(-m)) \ne 0$ and a section
$s \in H^{0}(\mathcal{E}(-m))$ with one-dimensional zero locus $Y \subset \mathbb{P}^{2}$.

Take a general line $Z \subset \mathbb{P}^{2}$ and let the
intersection $Z \cap Y$ consists of $k$ points. Since $c_{1}(\mathcal{E}(-m))=-2$, we have
$$
\mathcal{E}(-m)|_{Z} = \mathcal{O}_{\mathbb{P}^{1}}(k) \oplus \mathcal{O}_{\mathbb{P}^{1}}(-k-2).
$$
Then by Lemma~\ref{theorem:restriction-on-movable-curve} we get $2k + 2 \leqslant 3$.
Hence $k=0$ and $Y = \emptyset$, a contradiction.

Proposition~\ref{theorem:P-2-case} is completely proved.
\end{proof}

\begin{proposition}
\label{theorem:F-n-case}
In the above notation, we have $V \ne \mathbb{F}_{n}$ for $n \in \{0,2,3,4\}$.
\end{proposition}

\begin{proof}
Suppose that $V = \mathbb{F}_{n}$, $n \in \{0,2,3,4\}$.

We have $H^{4}(\mathbb{F}_{n}, \mathbb{Z}) \simeq \mathbb{Z}$ and
$H^{2}(\mathbb{F}_{n}, \mathbb{Z}) \simeq \mathbb{Z}\cdot h \oplus \mathbb{Z}\cdot l$ where $h$ is the negative section and $l$
is a fibre on $\mathbb{F}_{n}$.
Set $c_{1}: = ah + bl$, $c_{2}: = c$ where $a$, $b$, $c \in \mathbb{Z}$. Then by Lemma~\ref{theorem:estimate-for-1-class} we have
\begin{equation}
\label{equation:estimates-for-c-1}
\begin{array}{c}
a=c_{1} \cdot l \leqslant -3K_{V}\cdot l = 6, \qquad b = c_{1} \cdot (h+nl) \leqslant 3(2+n).
\end{array}
\end{equation}
Moreover, $a \geqslant 0$ and $b \geqslant na$ because $H$ is nef.

Let $p$, $q$ be the integers such that $a = 2p + a'$, $b = 2q + b'$ for some $a'$, $b' \in \mathbb{Z}$ with
$-2 \leqslant a', b' \leqslant -1$. Consider the twisted bundle $\mathcal{E}':=
\mathcal{E}\otimes\mathcal{O}_{\mathbb{F_{n}}}(-p-q)$. Set $c_{i}':=c_{i}(\mathcal{E}')$ to be the $i$-th Chern class of $\mathcal{E'}$,
$i=1$, $2$. We have
\begin{equation}
\label{equation:expression-for-c}
\begin{array}{c}
c'_{1} = a'h + b'l, \qquad c_{2}' = c+nap-aq-bp-np^{2}+2pq.
\end{array}
\end{equation}

Furthermore, from \eqref{equation:Rieman-Roch-formula} we obtain
\begin{equation}
\label{equation:RR-1}
\begin{array}{c}
\chi(V, \mathcal{E}) = -\frac{1}{2}na(a+1)+ab+a+b-c+2
\end{array}
\end{equation}

and
\begin{equation}
\label{equation:RR-2}
\begin{array}{c}
\chi(V, \mathcal{E}') = (b'-\frac{1}{2}na')(a'+1)+a'-c_{2}'+2.
\end{array}
\end{equation}

\begin{lemma}
\label{theorem:estimates-1}
In the above notation, if $n = 0$, then $c_{2}' < 0$ and $\chi(V, \mathcal{E}') > 0$.
\end{lemma}

\begin{proof}
By \eqref{equation:estimates-for-c-1} we have $0 \leqslant a, b \leqslant 6$. Furthermore, from
\eqref{equation:expression-for-c} and \eqref{equation:RR-1} we get
$$
c_{2}'=c-aq-bp+2pq = \frac{1}{2}ab+a+b-\chi(V, \mathcal{E})+2+\frac{1}{2}a'b'.
$$

By \eqref{equation:dimension-of-H} and Lemma~\ref{theorem:good-estimate} $\chi(V,\mathcal{E}) \in \{33,34,35,36\}$.
Then it is easy to see that $c_{2}' < 0$
when either $a$ or $b$ is less than $6$. Moreover, $c_{2}' < 0$ for all $0 \leqslant a, b \leqslant 6$  and $\chi(V, \mathcal{E}) \in \{35, 36\}$.

Suppose that $a=b=6$ and $\chi(V, \mathcal{E}) \in \{33, 34\}$. In particular, we have $c_{1} = 6h + 6l = -3K_{V}$. Then it follows from the proof of Lemma~\ref{theorem:estimate-for-1-class}
that
$$
-K_{W} \cdot D_{W} \cdot g^{*}B = 0
$$
for every curve $B$ on the surface $V$. From Lemma~\ref{theorem:when-anti-can-class-nef} we deduce that $-K_{W}$ is nef and hence $W$ is a weak Fano threefold.

On the other hand, from \eqref{equation:RR-1} we get
$$
\chi(V, \mathcal{E}) = ab+a+b-c+2
$$
Then for $a=b=6$ and $\chi(V, \mathcal{E}) \in \{33, 34\}$ we have $c \in \{16, 17\}$. From \eqref{equation:formula-for-degree} we get
$-K_{W}^{3} \in \{64, 56\}$ and hence
by the Riemann--Roch formula and Kawamata--Vieweg Vanishing Theorem
$$
\dim |-K_{W}| = -\frac{1}{2}K_{W}^{3} + 2 \leqslant 34.
$$
But according to Remark~\ref{remark:inequalities} $\dim |-K_{W}| \geqslant 35$, a contradiction. Thus, we get $c_{2}' < 0$.

Now, from \eqref{equation:RR-2} we get $$
\chi(V, \mathcal{E}') = b'(a'+1)+a'-c_{2}'+2.
$$
Since $-2 \leqslant a', b' \leqslant -1$ and $c_{2}' < 0$, we have $\chi(V, \mathcal{E}') > 0$.
\end{proof}

\begin{lemma}
\label{theorem:estimates-2}
In the above notation, if $n = 2$, then $c_{2}' < 0$ and $\chi(V, \mathcal{E}') > 0$.
\end{lemma}

\begin{proof}
By \eqref{equation:estimates-for-c-1} we have $2a \leqslant b \leqslant 12$ and $0 \leqslant a \leqslant 6$.
Furthermore, from \eqref{equation:expression-for-c} and \eqref{equation:RR-1} we get
$$
c_{2}'=c+2ap-aq-bp-2p^{2}+2pq = -\frac{1}{2}a^{2}+\frac{1}{2}ab+b-\chi(V, \mathcal{E})+
2-\frac{1}{2}a'^{2}+\frac{1}{2}a'b'.
$$

By \eqref{equation:dimension-of-H} and Lemma~\ref{theorem:good-estimate} $\chi(V,\mathcal{E}) \in \{33,34,35,36\}$.
Then for $2a \leqslant b \leqslant 12$ and $0 \leqslant a \leqslant 6$ we have

$$
c_{2}' \leqslant -\frac{1}{2}a^{2} + 6a + 14 - \chi(V, \mathcal{E}) - \frac{1}{2}a'^{2} - a' \leqslant
32 - \chi(V, \mathcal{E}) - \frac{1}{2}a'^{2} - a' \leqslant -1  - \frac{1}{2}a'^{2} - a' < 0.
$$

Finally, from \eqref{equation:RR-2} we get
$$
\chi(V, \mathcal{E}') = (b'-a')(a'+1)+a'-c_{2}'+2.
$$
Then $\chi(V, \mathcal{E}') \leqslant 0$ only for $a' = -2$ and
$b' = -1$ because $-2 \leqslant a', b' \leqslant -1$ and $c_{2}' < 0$.

But if $a' = -2$ and $b' = -1$, then $\chi(V, \mathcal{E}') \leqslant 0$ only for $c'_{2} = -1$. Moreover, for $b' = -1$ we
have $b \leqslant 11$. Thus, we get
$$
-1 = c_{2}' \leqslant -\frac{1}{2}a^{2} + \frac{11}{2}a + 12 - \chi(V, \mathcal{E}) < -1,
$$
a contradiction. Hence $\chi(V, \mathcal{E}') > 0$.
\end{proof}

\begin{lemma}
\label{theorem:estimates-3}
In the above notation, if $n = 3$, then $c_{2}' < 0$ and $\chi(V, \mathcal{E}') > 0$.
\end{lemma}

\begin{proof}
By \eqref{equation:estimates-for-c-1} we have $3a \leqslant b \leqslant 15$ and hence $0 \leqslant a \leqslant 5$.
Furthermore, from \eqref{equation:expression-for-c} and \eqref{equation:RR-1} we get
$$
c_{2}'=c+3ap-aq-bp-3p^{2}+2pq = -\frac{3}{4}a^{2}+\frac{1}{2}ab-\frac{1}{2}a+b-\chi(V, \mathcal{E})+
2-\frac{3}{4}a'^{2}+\frac{1}{2}a'b'.
$$

By \eqref{equation:dimension-of-H} and Lemma~\ref{theorem:good-estimate} $\chi(V,\mathcal{E}) \in \{33,34,35,36\}$.
If $\chi(V,\mathcal{E}) \in \{34,35,36\}$, then for $3a \leqslant b \leqslant 15$ and $0 \leqslant a \leqslant 5$ we have
$$
c_{2}' \leqslant -\frac{3}{4}a^{2} + 7a + 17 - \chi(V, \mathcal{E})-\frac{3}{4}a'^{2} + \frac{1}{2}a'b' <
33 - \chi(V, \mathcal{E}) - \frac{1}{2}b' \leqslant 0.
$$

If $\chi(V,\mathcal{E}) = 33$, then for $3a \leqslant b \leqslant 14$ and $0 \leqslant a \leqslant 5$ we have
$$
c_{2}' \leqslant -\frac{3}{4}a^{2} + \frac{13}{2}a - 17 - \frac{3}{4}a'^{2}+\frac{1}{2}a'b' \leqslant -\frac{3}{4}a^{2} + \frac{13}{2}a - 15  -\frac{3}{4}a'^{2} < 0.
$$

Furthermore, for $\chi(V,\mathcal{E}) = 33$, $b = 15$ and $0 \leqslant a \leqslant 4$ we have
$$
c_{2}' = -\frac{3}{4}a^{2} + 7a - 16 - \frac{3}{4}a'^{2} - \frac{1}{2}a' < 0.
$$

Suppose that $\chi(V,\mathcal{E}) = 33$, $a = 5$ and $b = 15$ (note that $c_{2}' = 0$ in this case).
Then from \eqref{equation:estimates-for-c-1} and the proof of Lemma~\ref{theorem:estimate-for-1-class} we
deduce that
$$
-K_{W} \cdot D_{W} \cdot g^{*}B = 0
$$
where $B \sim h + 3l$ on $V$. Lemma~\ref{theorem:when-anti-can-class-nef} implies that $-K_{W}$ is nef and hence $W$ is a weak Fano threefold.

On the other hand, from \eqref{equation:RR-1} we get $$
\chi(V, \mathcal{E}) = -\frac{3}{2}a(a+1) + ab+ a + b - c + 2.
$$
Then for $\chi(V,\mathcal{E}) = 33$, $a = 5$ and $b = 15$ we have $c = 19$. From \eqref{equation:formula-for-degree} we get $-K_{W}^{3} = 46$ and hence
by the Riemann--Roch formula and Kawamata--Vieweg Vanishing Theorem
$$
\dim |-K_{W}| = -\frac{1}{2}K_{W}^{3} + 2 = 25.
$$
But according to Remark~\ref{remark:inequalities} $\dim |-K_{W}| \geqslant 35$, a contradiction. Thus, we get $c_{2}' < 0$.

Now, from \eqref{equation:RR-2} we get
$$
\chi(V, \mathcal{E}') = (b'-\frac{3}{2}a')(a'+1)+a'-c_{2}'+2.
$$
Then $\chi(V, \mathcal{E}') \leqslant 0$ only for $a'=-2$ because  $-2 \leqslant a',b' \leqslant -1$ and $c_{2}' < 0$.

If $a' = -2$ and $b' = -1$, then $\chi(V, \mathcal{E}') = -2 - c_{2}'$ and
$$
c_{2}' = -\frac{3}{4}a^{2}+\frac{1}{2}ab-\frac{1}{2}a+b-\chi(V, \mathcal{E}).
$$
Since $a' = -2$, we get $0 \leqslant a \leqslant 4$. Then for $3a \leqslant b \leqslant 13$ we have
$$
c_{2}' \leqslant -\frac{3}{4}a^{2} + 6a + 13 - \chi(V, \mathcal{E}) \leqslant  -\frac{3}{4}a^{2} + 6a - 20 < -2.
$$

Furthermore, for $b > 13$ we get $b = 15$ because $b' = -1$. Then for $0 \leqslant a \leqslant 3$ we have

$$
c_{2}' \leqslant -\frac{3}{4}a^{2} + 7a + 15 - \chi(V, \mathcal{E}) \leqslant -\frac{3}{4}a^{2} + 7a - 18 < -2.
$$

Suppose that $a = 4$ and $b = 15$ (note that $c_{2}' = -2$ and hence $\chi(V, \mathcal{E}') = 0$ in this case).
Then from \eqref{equation:estimates-for-c-1} and the proof of Lemma~\ref{theorem:estimate-for-1-class} we
deduce that
$$
-K_{W} \cdot D_{W} \cdot g^{*}B = 0
$$
where $B \sim h + 3l$ on $V$. Lemma~\ref{theorem:when-anti-can-class-nef} implies that $-K_{W}$ is nef and hence $W$ is a weak Fano threefold.
On the other hand, for $a = 4$ and $b = 15$ from \eqref{equation:RR-1} we get $c = 51 - \chi(V, \mathcal{E})$ and hence $c \in \{15,16,17,18\}$.
Then by \eqref{equation:formula-for-degree} we have $-K_{W}^{3} \in \{48, 56, 64, 72\}$.

If $-K_{W}^{3} = 72$, then from Theorem~\ref{theorem:Prokhorov-degree},
Remarks~\ref{remark:K-trivial-contraction-1}, \ref{remark:unique-terminal-modification-1} and \ref{remark:unique-terminal-modification-2} we deduce that the threefold $W$ is isomorphic to that constructed either in Example~\ref{example:examp-1} or in  Example~\ref{example:examp-2}.
In particular, we have  either
$\rho(W) = 2$ or $5$ which is impossible because $\rho(W)  = 3$.

Furthermore, for $-K_{W}^{3} \in \{48, 56, 64\}$ by the Riemann--Roch formula and Kawamata--Vieweg Vanishing Theorem we have
$$ \dim |-K_{W}| = -\frac{1}{2}K_{W}^{3} + 2 \leqslant 34.
$$
But according to Remark~\ref{remark:inequalities} $\dim |-K_{W}| \geqslant 35$, a contradiction.

Finally, if $a' = -2$ and $b' = -2$, then $\chi(V, \mathcal{E}') = -1 - c_{2}'$ and
$$
c_{2}' = -\frac{3}{4}a^{2}+\frac{1}{2}ab-\frac{1}{2}a+b-\chi(V, \mathcal{E})+1.
$$
We have $a \leqslant 4$ and                         $3a \leqslant b \leqslant 14$. Then
$$
c_{2}' \leqslant -\frac{3}{4}a^{2} + \frac{13}{2}a - 18 < -1.
$$

Thus, we get $\chi(V, \mathcal{E}') > 0$.
\end{proof}

\begin{lemma}
\label{theorem:estimates-4}
In the above notation, if $n = 4$, then $c_{2}' < 0$ and $\chi(V, \mathcal{E}') > 0$.
\end{lemma}

\begin{proof}
By \eqref{equation:estimates-for-c-1} we have $4a \leqslant b \leqslant 18$ and hence $0 \leqslant a \leqslant 4$.
Furthermore, from \eqref{equation:expression-for-c} and \eqref{equation:RR-1} we get
$$
c_{2}'=c+4ap-aq-bp-4p^{2}+2pq = -a^{2}+\frac{1}{2}ab-a+b-\chi(V, \mathcal{E})+
2-a'^{2}+\frac{1}{2}a'b'.
$$

By \eqref{equation:dimension-of-H} and Lemma~\ref{theorem:good-estimate} $\chi(V,\mathcal{E}) \in \{33,34,35,36\}$.
If $b \leqslant 17$, then for $0 \leqslant a \leqslant 4$ we have
$$
c_{2}' \leqslant -a^{2} + \frac{15}{2}a - 14 - a'^{2} + \frac{1}{2}a'b' \leqslant
-a^{2} + \frac{15}{2}a - 14 - a'^{2} - a' < 0.
$$

Suppose that $b = 18$ and $\chi(V,\mathcal{E}) \in \{33,34,35\}$. Then $c_{2}' \geqslant 0$ only for $a \in \{3, 4\}$.
Furthermore, from \eqref{equation:estimates-for-c-1} and the proof of Lemma~\ref{theorem:estimate-for-1-class} we
deduce that
$$
-K_{W} \cdot D_{W} \cdot g^{*}B = 0
$$
where $B \sim h + 4l$ on $V$. Lemma~\ref{theorem:when-anti-can-class-nef} implies that $-K_{W}$ is nef and hence $W$ is a weak Fano
threefold.

On the other hand, for $a \in \{3, 4\}$ and $b = 18$ from \eqref{equation:RR-1} we get $$
c = a(17 - 2a) + 20 - \chi(V,\mathcal{E})
$$
and hence $17 \leqslant c \leqslant 23$. Since for $a \in \{3, 4\}$ and $b = 18$ we have $c_{1}^{2} \in \{72, 80\}$,
from \eqref{equation:formula-for-degree} we get
$$
-K_{W}^{3} \in \{192 - 8c, 208 - 8c\}.
$$

If $c = 17$ and $c_{1}^{2} = 80$, then $-K_{W}^{3} = 72$ and from Theorem~\ref{theorem:Prokhorov-degree},
Remarks~\ref{remark:K-trivial-contraction-1}, \ref{remark:unique-terminal-modification-1} and \ref{remark:unique-terminal-modification-2}
we deduce that the threefold $W$ is isomorphic to that constructed either in Example~\ref{example:examp-1} or in
Example~\ref{example:examp-2}. In particular, we have  either $\rho(W) = 2$ or $5$
which is impossible because $\rho(W)  = 3$.

Furthermore, for $18 \leqslant c \leqslant 23$ we have
$$
-K_{W}^{3} \leqslant 208 - 8c \leqslant 64 $$
and by the Riemann--Roch formula and Kawamata--Vieweg Vanishing Theorem $$
\dim |-K_{W}| = -\frac{1}{2}K_{W}^{3} + 2 \leqslant 34.
$$
But according to Remark~\ref{remark:inequalities} $\dim |-K_{W}| \geqslant 35$, a contradiction.

Suppose that $b = 18$ and $\chi(V,\mathcal{E}) = 36$.  Then again $c_{2}' \geqslant 0$ only for $a \in \{3,4\}$. By the previous arguments $W$ is a weak Fano
threefold such that $-K_{W}^{3} \leqslant 56$ and hence $\dim |-K_{W}|  \leqslant 30$ which contradicts the estimate in Remark~\ref{remark:inequalities}.   Thus, we get $c_{2}' < 0$.

Now, from \eqref{equation:RR-2} we obtain
$$
\chi(V, \mathcal{E}') = (b'-2a')(a'+1)+a'-c_{2}'+2.
$$
Then $\chi(V, \mathcal{E}') \leqslant 0$ only for $a' = -2$ because $-2 \leqslant a', b' \leqslant -1$ and $c_{2}' < 0$.

If $a' = -2$ and $b' = -1$, then $\chi(V, \mathcal{E}') = -3 - c_{2}'$ and
$$
c_{2}' = -a^{2}+\frac{1}{2}ab-a+b-\chi(V, \mathcal{E})-1.
$$
Since  $b' = -1$, we have $4a \leqslant b \leqslant 17$. If $0 \leqslant a \leqslant 3$, then we have
$$
c_{2}' \leqslant -a^{2} + \frac{15}{2}a - 17 < -3.
$$

If $a = 4$ and $4a \leqslant b \leqslant 16$, then we have
$$
c_{2}' \leqslant -54 + 3b < -3.
$$

Suppose that $a = 4$ and $b = 17$ (note that $c_{2}' = -3$ and hence $\chi(V, \mathcal{E}') = 0$ in this case).
Then we have $-K_{V} + c_{1} = 6h + 23l$ which is not nef on $V$. This contradicts Lemma~\ref{theorem:property-of-1-class}.

Finally, if $a' = -2$ and $b' = -2$, then $\chi(V, \mathcal{E}') = -2 - c_{2}'$ and
$$
c_{2}' = -a^{2}+\frac{1}{2}ab-a+b-\chi(V, \mathcal{E}).
$$

If $4a \leqslant b \leqslant 16$, then for $0 \leqslant a \leqslant 4$ we have
$$
c_{2}' \leqslant -a^{2} + 7a - 17 < -2.
$$

Furthermore, for $b > 16$ we get $b = 18$ because $b' = -2$. Then for $0 \leqslant a \leqslant 3$ we have
$0 \leqslant a \leqslant 2$, since $a' = -2$, and
$$
c_{2}' \leqslant -a^{2} + 8a - 15 < -2.
$$

Suppose that $b = 18$ and $a = 4$. Then, as in the proof of the
estimate $c_{2}'<0$, we get that $W$ is a weak Fano threefold such
that either $-K_{W}^{3} = 72$ or $-K_{W}^{3} \leqslant 64$. As we
have seen this is impossible. Thus, we get $\chi(V, \mathcal{E}')
> 0$.
\end{proof}

\begin{lemma}[{\cite[Claim 10.18]{Prokhorov-degree}}]
\label{theorem:non-zero-section}
In the above notation, we have $H^{0}(V, \mathcal{E}') \ne 0$.
\end{lemma}

\begin{proof}
Assume that $H^{0}(V, \mathcal{E}') = 0$. Then by Lemmas~\ref{theorem:estimates-1}--\ref{theorem:estimates-4} we have
$H^{2}(V, \mathcal{E}') \ne 0$. Furthermore, by Serre Duality
$$
H^{2}(V, \mathcal{E}')^{*} \simeq H^{0}(V, \mathcal{E}'^{*} \otimes K_{V})
\simeq H^{0}(V, \mathcal{E}' \otimes \det \mathcal{E}'^{*}\otimes K_{V}).
$$
But $(\det \mathcal{E}'^{*}\otimes K_{V})^{*} = \mathcal{O}_{V}((a'+2)h + (b'+n+2)l)$ and hence
$H^{0}(V, (\det \mathcal{E}'^{*}\otimes K_{V})^{*}) \ne 0$, a contradiction.
\end{proof}

Let $s \in H^{0}(V, \mathcal{E}')$ be a nonzero section. If $s$ does not vanish anywhere, then $\mathcal{E}'$ is an extension
of some line bundle $\mathcal{E}_{1}$ by $\mathcal{O}$. But then $c_{2}' = 0$ which is impossible by
Lemmas~\ref{theorem:estimates-1}--\ref{theorem:estimates-4}.

Let $\emptyset \ne Y \subset V$ be the zero locus of $s$. Since $c_{2}' < 0$, we have $\dim Y = 1$. Let
$Y \sim q_{1}h + q_{2}l$. Then the restriction of $\mathcal{E}'$ on general curve $Z \in |h+nl|$ is
$$
\mathcal{E}'|_{Z}=\mathcal{O}_{\mathbb{P}^{1}}(q_{2})\oplus \mathcal{O}_{\mathbb{P}^{1}}(b'-q_{2}).
$$
By Lemma~\ref{theorem:restriction-on-movable-curve} applied to $\mathcal{E}'$
we have
$$
0 \leqslant 2q_{2}-b' = c_{1}' \cdot Z = b' \leqslant -1,
$$
a contradiction. Proposition~\ref{theorem:F-n-case} is completely proved.
\end{proof}

From Propositions~\ref{theorem:P-2-case} and
\ref{theorem:F-n-case} we derive a contradiction in case when $V$
is a surface which completes the proof of
Theorem~\ref{theorem:main}.


\begin{thebibliography}{31}
\bibitem{Bayle}
Bayle L. Classification des vari\'et\'es complexes projectives de
dimension trois dont une section hyperplane g\'en\'erale est une
surface d'Enriques // J. Reine Angew. Math. V. 449. 1995. P.
9--63.

\bibitem{Borisov-boundedness}
Borisov A. Boundedness of Fano threefolds with log-terminal singularities of given index // J. Math. Sci. Univ. Tokyo.
2001. V. 8. P. 329-342.

\bibitem{Cheltsov-Enriques-1}
Cheltsov I. A. Rationality of Enriques-Fano threefold of genus
five // Russ. Acad. Sci. Izv. Math. 2004. V. 68(3). P. 181-194.

\bibitem{Cheltsov-Enriques}
Cheltsov I. A. Singularity of three-dimensional manifolds possessing an ample effetive divisor -- a smooth surface of Kodaira dimension zero //
Mat. Zametki. V. 59(4). 1996. P. 618--626.

\bibitem{Conte}
Conte A. Two examples of algebraic threefolds whose hyperplane sections are Enriques surfaces // Algebraic geometry -- open problems (Ravello, 1982);
Lecture Notes in Math. (Springer, Berlin, 1983). V. 997. P. 124--130.

\bibitem{Conte-Murre}
Conte A., Murre J. P. Algebraic varieties of dimension three  whose hyperplane sections are Enriques surfaces // Ann. Scuola Norm. Sup. Pisa CL. Sci.
V. 12(1). 1985. P. 43--80.

\bibitem{Cutkosky}
Cutkosky S. Elementary contractions of Gorenstein threefolds //
Math. Ann. 1988. V. 280. P. 521-525.

\bibitem{Dolgachev}
Dolgachev I. V. Weighted projective varieties // Lecture Notes in Math. 1982. V. 956. P. 34-71.

\bibitem{Fano}

Fano G. Sulle varieta algebriche a tre dimensione le cui sezioni iperpiane sono superficie di genere zero e bigenere uno // Mem. Mat. Sci. Fis. Natur.
Soc. Ital. Sci. Ser. 3. 1938. P. 41--66.

\bibitem{Iano-Fletcher}
Iano-Fletcher A. R. Working with weighted complete intersections // Explicit Birational Geometry of 3-Folds (A. Corti and
M. Reid, eds.). London Math. Soc. Lecture Note Sec. Cambridge Univ. Press. Cambridge 2000. V. 281. P. 101-173.

\bibitem{Iskovskikh-anti-canonical-models}
Iskovskikh V. A. Anticanonical models of three-dimensional algebraic varieites // English trans. in J. Soviet Math. 1980. V. 13.
P. 745-814.

\bibitem{VA-1}
Iskovskikh V. A. Fano threefolds I // Math. USSR. Izv. 1977. V. 11. P. 485-527.

\bibitem{VA-2}
Iskovskikh V. A. Fano threefolds II // Math. USSR. Izv. 1978. V. 12. P. 469-506.

\bibitem{Jahnke-Radloff}
Jahnke P., Radloff I. Gorenstein Fano threefolds with base points in
the anticanonical system // arXiv: math. AG0404156 (2005).

\bibitem{Kawamata}
Kawamata Y. The crepant blowing-up of 3-dimensional canonical singularities and its application to the degeneration of
surfaces //Ann. of Math. 1988. V. 127(2). P. 93-163.

\bibitem{KLM}
Knutsen A. L., Lopez A. F., Munoz R. On the extendability of projective surface and a genus bound for Enriques-Fano
threefolds // arXiv: math. AG0605750 (2006).

\bibitem{Kollar-flops}
Koll\'ar J. Flops // Cambridge, Massachusetts: Harvard Univ. 1987.(Preprint).

\bibitem{Kollar-Mori}
Koll\'ar J., Mori S. Birational geometry of algebraic varieties
// Cambridge Univ. Press. 1998.

\bibitem{Kreuzer-Skarke}
Kreuzer M., Skarke H. PALP: A package for analyzing lattice polytopes with applications to toric geometry //
Computer Phys. Comm. 2004. V. 157. P. 87-106.

\bibitem{Lvovsky}
Lvovskij S. M. The boundedness of the degree of Fano threefolds // English trans. Math. USSR Izv. 1982. V. 19. P. 521-558.

\bibitem{Mori-flip}
Mori S. Flip theorem and the existence of minimal models for 3-folds // J. Amer. Math. Soc. 1988. V. 1. P. 117-253.

\bibitem{Mori-Mukai}
Mori S., Mukai S. Classification of Fano 3-folds with $B_{2} \ge 2$ // Manuscr. Math. 1981. V. 36. P. 147-162.

\bibitem{Namikawa-smoothing}
Namikawa Y. Smoothing Fano 3-folds // J. Algebraic Geom. 1997. V. 6. P. 307-324.

\bibitem{CPS}
Przyjalkowsky V. V., Cheltsov I.A., Shramov K. A. Hyperelliptic and triginal Fano threefolds // Russ. Acad. Sci. Izv. Math.
2005. V. 69(2). P. 145-204.

\bibitem{Prokhorov-Enriques-1}
Prokhorov Yu. G. On algebraic threefolds whose hyperplane sections are Enriques surfaces // Russ. Acad. Sci. Sb. Math. 1995. V. 186 (9). P. 1341-1352.

\bibitem{Prokhorov-degree}
Prokhorov Yu. G. On the degree of Fano threefolds with canonical Gorenstein singularities // Russ. Acad. Sci. Sb. Math.
2005. V. 196(1). P. 81-122.

\bibitem{Prokhorov-Enriques}
Prokhorov Yu. G. On Fano-Enriques threefolds // Russ. Acad. Sci. Sb. Math. 2007. V. 198(4). P. 559-574.

\bibitem{Reid-canonical-threefolds}
Reid M. Canonical 3-folds // Algebraic Geometry, Angers. 1979. P. 273-310.

\bibitem{Reid-chapters-on-surfaces}
Reid M. Chapters on algebraic surfaces // Complex algebraic geometry
(Park City. Ut. 1993) P. 3-159.

\bibitem{Reid-morphisms-Kawamata}
Reid M. Projective morphisms according to Kawamata // Preprint Univ.
of Warwick. 1983.

\bibitem{Sano}
Sano T. Classification of non-Gorenstein $\mathbb{Q}$-Fano
$3$-folds of index $1$ // J. Math. Soc. Japan. V. 47(2). 1995. P.
369--380.
\end{thebibliography}
\end{document}